\documentclass[11pt]{amsart}

\usepackage{hyperref}

\usepackage{graphicx}

\usepackage{xcolor}

\usepackage{amsmath, amssymb, amsfonts, amsthm, fullpage}
\usepackage{graphics,graphicx}
\usepackage{accents}
\usepackage{color}

\usepackage{algorithm}
\usepackage{algorithmicx}
\usepackage{algpseudocode}

\usepackage[margin=0.9in]{geometry}

\newcommand{\norma}[1]{{\left\vert\kern-0.25ex\left\vert\kern-0.25ex\left\vert #1
    \right\vert\kern-0.25ex\right\vert\kern-0.25ex\right\vert}}
\newcommand{\sech}{\text{sech}}

\newcommand{\tbu}{\tilde{\mathbf{u}}}

\newcommand{\bd}{\mathbf{d}}
\newcommand{\by}{\mathbf{y}}
\newcommand{\teta}{\tilde{\eta}}
\newcommand{\tphi}{\tilde{\phi}}

\newcommand{\tw}{\tilde{w}}
\newcommand{\tv}{\tilde{v}}
\newcommand{\tf}{\tilde{f}}

\newcommand{\tz}{\tilde{z}}

\newcommand{\Alpha}{\mathrm{A}}
\newcommand{\Beta}{\mathrm{B}}

\newcommand{\bo}{\mathbf{0}}
\newcommand{\btf}{\mathbf{f}}

\newcommand{\bx}{\mathbf{x}}
\newcommand{\bF}{\mathbf{F}}
\newcommand{\bn}{\mathbf{n}}
\newcommand{\bu}{\mathbf{u}}
\newcommand{\bw}{\mathbf{w}}
\newcommand{\bz}{\mathbf{z}}
\newcommand{\bv}{\mathbf{v}}
\newcommand{\bA}{\mathbf{A}}

\newcommand{\bH}{\mathbf{H}}

\newcommand{\bJ}{\mathbf{J}}

\newcommand{\bU}{\mathbf{U}}

\newcommand{\bI}{\mathbf{I}}

\newcommand{\balpha}{\boldsymbol{\alpha}}
\newcommand{\bgamma}{\boldsymbol{\gamma}}

\newcommand{\bchi}{\boldsymbol{\chi}}

\newcommand{\bxi}{\boldsymbol{\xi}}

\newcommand{\Div}{\nabla\!\cdot\!}
\newcommand{\Curl}{\nabla\!\times\!}

\newcommand{\tbn}[1]{{\left\vert\kern-0.25ex\left\vert\kern-0.25ex\left\vert #1 \right\vert\kern-0.25ex\right\vert\kern-0.25ex\right\vert}}

\newtheorem{remark}{Remark}[section]

\newtheorem{proposition}{Proposition}[section]
\newtheorem{theorem}{Theorem}[section]
\newtheorem{corollary}{Corollary}[section]

\AtBeginDocument{%
  \setlength{\belowdisplayskip}{5pt} \setlength{\belowdisplayshortskip}{5pt}
  \setlength{\abovedisplayskip}{5pt} \setlength{\abovedisplayshortskip}{5pt}
}

\title{Bona-Smith-type systems in bounded domains with slip-wall boundary conditions: Theoretical justification and a conservative numerical scheme}

\author{Dimitrios Antonopoulos}
\address{\textbf{D.~Antonopoulos:} University of Athens, Department of Mathematics, 15784 Zographou, Greece}
\email{antonod@math.uoa.gr}

\author{Dimitrios Mitsotakis}
\address{\textbf{D.~Mitsotakis:} Victoria University of Wellington, School of Mathematics and Statistics, PO Box 600, Wellington 6140, New Zealand}
\email{dimitrios.mitsotakis@vuw.ac.nz}


\date{\today}


\begin{document}

\begin{abstract}
Considered herein is a class of Boussinesq systems of Bona-Smith type that describe the propagation of long surface water waves of small amplitude in bounded two-dimensional domains with slip-wall boundary conditions and variable bottom topography. Such boundary conditions are necessary in situations involving water waves in channels, ports, and generally in basins with solid boundaries. We prove that, given appropriate initial conditions, the corresponding initial-boundary value problems have unique solutions locally in time, which is a fundamental property of deterministic mathematical modeling. Moreover, we demonstrate that the systems under consideration adhere to three basic conservation laws for water waves: mass, vorticity, and energy conservation.

The theoretical analysis of these specific Boussinesq systems leads to a conservative mixed finite element formulation. Using explicit, relaxation Runge-Kutta methods for the discretization in time, we devise a fully discrete scheme for the numerical solution of initial-boundary value problems with slip-wall conditions, preserving mass, vorticity, and energy. Finally, we present a series of challenging numerical experiments to assess the applicability of the new numerical model.
\end{abstract}

\maketitle

\section{Introduction}\label{sec:intro}

The fundamental equations describing nonlinear and dispersive water waves are the inviscid and incompressible Euler equations of fluid mechanics, accompanied by appropriate boundary conditions \cite{Whitham}. Due to the immense difficulty of studying the Euler equations of water wave theory theoretically and numerically, simplified mathematical models have been proposed as alternative systems to describe water waves in various regimes \cite{Lannes2013}. One such regime of particular interest is the small amplitude long wave regime without surface tension. Examples of small amplitude long water waves include tsunamis, solitary waves, and undular bores. However, the focus has primarily been on one-dimensional unidirectional models with flat bottom topography for the study of traveling waves \cite{Whitham}.

Boussinesq's influential work \cite{Bous1871,Bous1872} laid the foundations of asymptotic water wave models and the study of unidirectional wave propagation, such as the Korteweg-de Vries (KdV) and Benjamin-Bona-Mahony (BBM) equations \cite{KDV1895, BBM1972}, leading to a profound understanding of the dynamics of solitary waves. The need to study practical applications and more complex wave behaviors led to bidirectional water wave models, typically referred to as Boussinesq systems.

One category of asymptotically derived Boussinesq systems for bidirectional wave propagation is the family of the so-called $abcd$-Boussinesq systems. These systems were originally derived in one-dimension with the assumption of horizontal bottom topography \cite{BCS2002, BCS2004, BS} and generalized to two dimensions again with flat bottom \cite{BCL2005}. These systems are notable because they have been asymptotically justified \cite{BCL2005,Lannes2013}. However, not all of them are suitable for studying water wave problems. For instance, the so-called KdV-KdV system does not exhibit classical solitary waves \cite{BDM2007i,BDM2008ii}, violating physical laws, while other alternatives in this category violate mathematical laws by not adhering to the deterministic law of well-posedness, such as the Kaup-Boussinesq system \cite{ABM2019}. Furthermore, none of the systems presented in \cite{BCL2005} could be used directly in bounded domains with slip-wall boundary conditions unless we modify their dispersive terms. For a review of some of the fundamental properties of these systems, we refer to \cite{DM2008}.

In the case of variable bottom topography, several Boussinesq systems have been derived and employed in various applications. Boussinesq systems with variable bottom topography include the Peregrine system \cite{P1967}, a BBM-BBM-type system \cite{Mitsotakis2009}, and Nwogu's system \cite{Nwogu93}. Although these systems are asymptotically equivalent, only the BBM-BBM-type system of \cite{Mitsotakis2009} has been proven well-posed in bounded domains with non-slip wall boundary conditions thus far \cite{DMS2009}. For the Peregrine and Nwogu systems, results are only available in one dimension with reflective boundary conditions \cite{Adamy2011, FP2005, MM2023}. In addition to the lack of theoretical justification, none of the previously mentioned systems preserve any reasonable form of energy in bounded domains with variable bottom, thus constructing energy-preserving numerical schemes is beyond the scope of these systems. It is noteworthy that all these systems, when reduced to the case of horizontal bottom topography and in one dimension, are included in the family of $abcd$-Boussinesq systems of \cite{BCS2002}.

In addition to the $abcd$-Boussinesq systems, there are numerous other approximations of the Euler equations that describe long water waves. These approximations include the Serre-Green-Naghdi equations \cite{S1953I,S1953II,SG1969,GN1976}, improved Serre-Green-Naghdi equations \cite{CDM2017,CDM2024}, the Madsen-Murray-S{\o}rensen system \cite{MMS1991}, non-hydrostatic shallow water approximations \cite{BGSM2011,BMSMS2015,WMKB2020},  hyperbolic approximations of water wave systems \cite{ALB2009,FG2017,D2019} as well as other systems with various enhancements \cite{GKW2000,MS1992,MMS1991,SM1995,MBL2002}. These high-order systems, some with enhanced dispersion characteristics, warrant further study. 

A new generalization of the family of $abcd$-Boussinesq systems of \cite{BCS2002,BCL2005} (see also \cite{chazel1, chazel2}) for variable bathymetry in two dimensions with small variations was proposed by \cite{IKM2023,IKKM2021}. By denoting $\bx=(x,y)\in\mathbb{R}^2$, $t\geq 0$ as the spatial and temporal independent variables respectively,  $D=D(\bx)$ as the bathymetry measured from the zero surface level, $\eta=\eta(\bx,t)$ as the free-surface elevation above the undisturbed zero-level of the water, and $\bu=\bu(\bx,t)$ as the horizontal velocity of the fluid at depth $z_\theta=-D+\theta(\eta+D)$ for some $\theta\in[0,1]$, this family of $abcd$-Boussinesq systems can be expressed in dimensional and unscaled variables as follows:
\begin{equation}\label{eq:abcd}
\begin{aligned}
& \eta_t+\Div[(D+\eta)\bu]-\Div\left\{a\nabla(D^3\Div\bu)+bD^2\nabla\eta_t\right\}=0\ ,\\
& \bu_t+g\nabla\eta+\tfrac{1}{2}\nabla|\bu|^2 -\nabla\left\{cg\Div(D^2\nabla\eta)+d\Div(D^2\bu_t)\right\}=0\ ,
\end{aligned}
\end{equation}
where
\begin{equation}\label{eq:abcdcoefs}
a=\frac{1}{2}\left(\frac{1}{3}-\theta^2\right)\mu,\quad b=\frac{1}{2}\left(\theta^2-\frac{1}{3}\right)(1-\mu), \quad c=\frac{1}{2}(\theta^2-1)\nu, \quad d= \frac{1}{2}(1-\theta^2)(1-\nu)\ ,
\end{equation}
for fixed $\mu,\nu\in \mathbb{R}$ and $0\leq \theta\leq 1$. The parameters $\mu$ and $\nu$ do not have any physical interpretation and are bi-products of the process of the asymptotic derivation while $\theta$ is related with the depth where we compute the horizontal velocity $\bu$ (see the Appendix for a derivation). In the previous notation, $\nabla$ is the gradient operator with respect to the spatial variable $\bx$. These systems can also incorporate moving bottom topography \cite{IKM2023}, but for the purposes of this work, we assume that the bottom is stationary. Furthermore, these systems preserve the vorticity, and when $b=d>0$, they preserve a reasonable form of energy \cite{IKKM2021}.

In this work, we investigate a specific subset of systems (\ref{eq:abcd}). In particular, we study the case $b=d=\frac{3\theta^2-1}{6}>0$ and $c=\frac{3\theta^2-2}{3}>0$ for $\theta^2\in\left(2/3,1\right)$, including the limiting cases $\theta^2=2/3$ and $\theta^2=1$. These systems can be derived from (\ref{eq:abcd})-(\ref{eq:abcdcoefs}) by setting $\mu=0$ and $\nu=\frac{4-6\theta^2}{3(1-\theta^2)}$. They are of particular interest due to their suitable physical and mathematical properties for analyzing small amplitude long water waves. They are also asymptotically justified \cite{IKM2023} compared to the Euler equations, possess classical traveling wave solutions \cite{Chen1998,Chen2000,DM2004,DM2008}, and maintain a reasonable form of total energy. Reducing these systems in one dimension and with a flat bottom topography, they coincide with the Bona-Smith systems of \cite{BCS2002,BS}, and for this reason, we will refer to them as Bona-Smith systems. It is worth mentioning that the limiting case $\theta^2=2/3$, where $c=0$ and $b=d=1/6$, is called BBM-BBM system (or regularized shallow water waves equations) studied in \cite{IKKM2021,KMS2020}. The limiting case $\theta^2=1$ is called classical Bona-Smith system (\cite{BS}) and is analyzed in this work.

It is noteworthy that irrotationality is an intrinsic component of the derivation of the $abcd$-Boussinesq systems of \cite{BCL2005} (including Peregrine's system \cite{P1967} as well as other systems in the same regime). Even if these Boussinesq  systems do not preserve the vorticity, these can be converted to asymptotically equivalent systems of the form (\ref{eq:abcd}), which preserves the vorticity, \cite{KMS2020,IKKM2021,IKM2023}. This is aligned with the observations in \cite{LL2013}, where it is stated that even if realistic flows are not potential, vorticity can be negligible in initially irrotational flows or in flows involving long waves of small amplitude (such as solitary and cnoidal waves). Realistic free-surface flows, however, might not remain irrotational even if they satisfy irrotationality at an initial stage. For example, some water waves break under the influence of bottom topography. Therefore, the validity of Boussinesq systems is limited to non-breaking waves \cite{P1967}. Mathematical modelling of water waves in the presence of vorticity has also been carried out in a different context cf. e.g. \cite{CL2014}.

In this work, we assume irrotational flow and we concentrate our attention on  Bona-Smith-type systems posed in a bounded domain $\Omega$ with slip-wall boundary conditions, where we typically assume that the boundary $\partial\Omega$ of $\Omega$ is sufficiently smooth. Such boundary conditions are required to study waves in ports, channels as well as interactions of long waves waves with marine structures. We assume initial conditions $\eta(\bx,0)=\eta_0(\bx)$ and $\bu(\bx,0)=\bu_0(\bx)$, where the initial velocity field is irrotational, i.e., $\Curl\bu_0(\bx)=0$ in $\Omega$. This guarantees the existence of an initial velocity potential $\phi_0$ such that $\bu_0(\bx)=\nabla\phi_0(\bx)$ in $\Omega$. After establishing the existence and uniqueness of solutions for these Bona-Smith systems, we introduce an energy-preserving fully discrete scheme for approximating these solutions. Such numerical methods ensure accurate results in long-time simulations and, in general, are more accurate than non-conservative methods.

Traditionally, Galerkin discretizations applied to problems in bounded domains with slip-wall boundary conditions incorporate Nitsche's method \cite{Nitche}. In the context of the BBM-BBM system with parameters $a=c=0$ and $b=d=1/6$, the standard Galerkin finite element method of \cite{KMS2020,IKKM2021} has been successfully employed alongside with Nitsche's method. However, the application of Nitsche's method for slip-wall boundary conditions presents challenges to energy conservation, necessitating intricate techniques such as solving large nonlinear systems for both the solution and the corresponding Lagrange multipliers.

In contrast, the numerical method proposed in this work operates on a formulation of the Bona-Smith systems based on the horizontal velocity potential $\phi$, where $\mathbf{u}=\nabla\phi$. This approach reduces the problem's dimensionality and facilitates a stable numerical discretization that circumvents the need for applying Nitsche's method to handle slip-wall boundary conditions \cite{Nitche}. This numerical method extends the modified Galerkin method and the idea of the discrete Laplacian of \cite{DMS2010} and results in a mixed formulation \cite{BBF2013} that allows the use of linear elements regardless the presence of a high-order derivative term. The numerical solutions of the proposed method preserve the mass, vorticity and energy functionals. For the temporal discretization we employ a fourth-order explicit relaxation Runge-Kutta method that respects the mass and energy conservation laws. 

Importantly, it is worth noting that the proposed numerical method can also be applied to solve the BBM-BBM system described in \cite{KMS2020,IKKM2021}, as it represents a special case of the Bona-Smith systems when $\theta^2= 2/3$.

The structure of this paper is as follows: In Section \ref{sec:exist}, we establish the existence and uniqueness of solutions to the initial-boundary value problem of the Bona-Smith systems with slip-wall boundary conditions. This makes the new Bona-Smith system the first $abcd$-Boussinesq system with variable bottom justified with slip-wall boundary conditions in two dimensions. Additionally, we review its basic properties such as conservation laws, linear dispersion characteristics, and the existence of line solitary waves. In Section \ref{sec:modgal}, we present a conservative modified Galerkin method for the numerical discretization of the aforementioned initial-boundary value problem. We provide details on the choice of initial conditions and present the explicit relaxation Runge-Kutta method that we employ for the discretization in time. We conclude this paper with numerical experiments. In Section \ref{sec:expval}, after computing experimental convergence rates for some exact solutions, we test the numerical model against standard benchmarks.

\section{Justification of the mathematical model}\label{sec:exist}

We consider the Bona-Smith system of Boussinesq equations, \cite{IKKM2021}, written in the form
\begin{equation}\label{eq:BS}
\begin{aligned}
&\eta_t+\Div [(D+\eta)\bu]-b\Div\left\{D^2\nabla\eta_t\right\}=0\ ,\\
&\bu_t+g\nabla\eta+\tfrac{1}{2}\nabla|\bu|^2-\nabla\left\{cg\Div(D^2\nabla\eta)+b\Div(D^2\bu_t)\right\}=0\ ,
\end{aligned}
\end{equation}
where
\begin{equation}\label{eq:coeffs}
b=\frac{3\theta^2-1}{6}>0, \quad c=\frac{3\theta^2-2}{3}> 0 \quad \text{and}\quad  \frac{2}{3}<\theta^2\leq 1\ .
\end{equation}
The limiting case where $\theta^2=2/3$ corresponds to the so-called BBM-BBM system (or regularized shallow-water equations) which has been studied in \cite{IKKM2021}. Moreover, the Bona-Smith systems (\ref{eq:BS}) can be used to study waves in a bounded domain $\Omega$ where the boundary consists of solid walls with no friction (slip-walls). Slip-walls can be described by the boundary conditions
\begin{equation}\label{eq:bcs}
\nabla\eta\cdot\bn=0\quad \text{and}\quad \bu\cdot\bn=0 \quad \text{on $\partial\Omega$}\ .
\end{equation}
Note that the solutions of the Euler equations satisfy the same conditions on a slip-wall \cite{IKM2023}, and therefore such boundary conditions are justified.

In what follows, we will study the initial-boundary value problem comprised the system (\ref{eq:BS}), the boundary conditions (\ref{eq:bcs}), and the initial conditions:
\begin{equation}\label{eq:ics}
\eta(\bx,0)=\eta_0(\bx)\quad \text{and}\quad \bu(\bx,0)=\bu_0(\bx)\quad \text{for $\bx\in\Omega$}\ .
\end{equation}
This serves as a model initial-boundary value problem for studying the propagation of long nonlinear waves of small amplitude in a bounded region.

\subsection{Energy conservation and Hamiltonian structure}

As demonstrated in \cite{IKM2023}, the solutions of the initial-boundary value problem (\ref{eq:BS})--(\ref{eq:bcs}) preserve the mass  
\begin{equation}\label{eq:mass}
\mathcal{M}(t;\eta)=\int_{\Omega}\eta~d\bx\ ,
\end{equation}
the vorticity
\begin{equation}\label{eq:vorticity}
\mathcal{V}(t;\bu)=\int_\Omega \Curl \bu ~d\bx\ ,
\end{equation}
and the energy functional
\begin{equation}\label{eq:energy}
\mathcal{E}(t;\eta,\bu)=\frac{1}{2}\int_\Omega g\eta^2+(D+\eta)|\bu|^2+cgD^2|\nabla\eta|^2 ~d\bx\ .
\end{equation}
Note that the integral in (\ref{eq:vorticity}) is not necessary since the vorticity $\Curl\bu$ is preserved pointwise. Also, in the case of general bottom topography, the so-called impulse functional is not preserved since the equations are not invariant under horizontal translations -- a property that holds true only when the bottom is flat. Note that the term impulse is used for quadratic invariants of water waves systems, and should not be confused with the momentum conservation \cite{Benjamin1984}.

Finally, we mention that the system (\ref{eq:BS}) is Hamiltonian: Writing the energy density as 
\begin{equation}
E(\eta,\bu)=\tfrac{1}{2}\left(g\eta^2+(D+\eta)|\bu|^2-cg\Div(D^2\nabla\eta) \eta \right)\ ,
\end{equation}
then the problem (\ref{eq:BS}) has the Hamiltonian structure:
\begin{equation}\label{eq:hamilton}
\bw_t=
\bJ\mathcal{D}\mathcal{E}(\bw)\ ,
\end{equation}
where $\bw=(\eta,\bu)^T$, $\mathcal{D}$ denotes the variational derivative, and $\bJ$ is the skew-adjoint operator
$$\bJ=\begin{pmatrix}
    0 & [I-b\Div(D^2\nabla )]^{-1}\Div\\
    [I-b\Div(D^2\nabla )]^{-1}\nabla & 0
\end{pmatrix} \ .$$

\subsection{Existence}

Here we establish the existence and uniqueness of solutions to the initial-boundary value problem (\ref{eq:BS}), (\ref{eq:bcs}), (\ref{eq:ics}).

Under the irrotationality assumption $\Curl\bu_0(\bx)=0$ of the initial velocity field and the fact that the velocity $\bu$ of the system (\ref{eq:BS}) satisfies $\Curl\bu(\bx,t)=\Curl\bu(\bx,0)$ for all $t\geq 0$, there is a velocity potential $\phi(\bx,t)$ such that  $\bu=\nabla\phi$. The corresponding initial-boundary value problem for the velocity potential can be written as: 
\begin{equation}\label{eq:BS2}
\begin{aligned}
&\begin{aligned}
&\eta_t+\Div [(D+\eta)\nabla\phi]-b\Div\left\{D^2\nabla\eta_t\right\}=0\ ,\\
&\phi_t+g\eta+\tfrac{1}{2}|\nabla\phi|^2-cg\Div(D^2\nabla\eta)-b\Div(D^2\nabla\phi_t)=0\ ,
\end{aligned}\quad \text{in $\Omega$}\ ,\\
&\eta(\bx,0)=\eta_0(\bx)\quad \text{and}\quad \phi(\bx,0)=\phi_0(\bx)\ ,\\
&\nabla\eta\cdot\bn=0\quad \text{and}\quad \nabla\phi\cdot\bn=0\quad \text{on $\partial\Omega$}\ .
\end{aligned}
\end{equation}
Note that if $(\eta,\phi)$ is a solution of the problem (\ref{eq:BS2}), then $(\eta,\bu)=(\eta,\nabla\phi)$ is a solution of the problem (\ref{eq:BS}), (\ref{eq:bcs}) with initial conditions $\eta(\bx,0)=\eta_0(\bx)$ and $\bu(\bx,0)=\bu_0(\bx)=\nabla\phi_0(\bx)$. Thus, we first establish the existence of solutions to (\ref{eq:BS2}). 

For simplicity of the notation we assume $g=1$. Moreover, for functions $u,v$, we will write $u \lesssim v$ if and only if $u\leq C v$ for some constant $C>0$. We will denote by $W^{k,p}=W^{k,p}(\Omega)$ the usual Sobolev space with $k$-times weakly differentiable functions in $\Omega$, and its corresponding norm $\|\cdot\|_{k,p}$. In the case $p=2$, we will denote by $H^k$ the spaces $W^{k,2}$ and the corresponding norms by $\|\cdot\|_k$. The index $k$ will be omitted when $k=0$ and $p=2$, i.e., when $H^0=L^2$.

For the analysis of this work, we assume that the depth function $D(\bx)\in W^{2,\infty}$, and that there are $D_m,D_M,D_m'$ and $D_M'$ such that
\begin{equation}\label{eq:depthcondition}
0<D_m\leq D(\bx)\leq D_M\quad \text{and}\quad 0<D_m'\leq |\nabla D(\bx)|\leq D_M'\quad \text{for all $\bx\in\Omega$}\ .
\end{equation}
Moreover, we assume that the domain $\Omega$ is of class $\mathcal{C}^{1,1}$ or smoother.

We define the bilinear forms $A:H^1\times H^1\to \mathbb{R}$ and $B:H^1\times H^1\to \mathbb{R}$
\begin{equation}
A_d(u,v)=(u,v)+d(D\nabla u,D\nabla v)\quad \text{and}\quad B_d(u,v)=(Du,Dv)+d(D^2\nabla u,D^2\nabla v)\ ,
\end{equation}
where $d>0$, and
$$(u,v)=\int_\Omega uv~ d\bx\ ,$$
the usual $L^2$ inner product. Because of the condition (\ref{eq:depthcondition}), these two bilinear forms are coercive and continuous.

We define the weak solution of (\ref{eq:BS2}) to be any solution $(\eta,\phi)\in H^2\times H^2$ that satisfies the problem
\begin{equation}\label{eq:weakform}
\begin{aligned}
& A_b(\eta_t,w)-((D+\eta)\nabla\phi,\nabla w)=0 \quad \text{for all $w\in H^2$}\ ,\\
& B_b(\phi_t,v)+B_c(\eta,v)+\tfrac{1}{2}(D^2|\nabla\phi|^2,v)=0 \quad \text{for all $v\in H^2$}\ ,\\
&\eta(\bx,0)=\eta_0(\bx)\quad \text{and}\quad \phi(\bx,0)=\phi_0(\bx)\ .
\end{aligned}
\end{equation}

\begin{proposition}\label{prop:existence}
If $(\eta_0,\phi_0)\in H^2\times H^2$, then
there is a time $T>0$ such that the problem (\ref{eq:weakform}) has a unique weak solution $(\eta,\phi)\in C^1([0,T];H^2\times H^2)$.
\end{proposition}

\begin{proof}
We define the mapping $f:L^2\to H^1$ such that for $u\in L^2$
\begin{equation}
    A_b(f[u],w)=(u,\nabla w)\quad\text{for all $w\in H^1$}\ .
\end{equation}
Taking $w=f[u]$ we see that $\|f[u]\|_1\lesssim \|u\|$. Moreover, from the theory of elliptic problems, we have that if $u\in H^1$, then $f[u]\in H^2$ and $\|f[u]\|_2\lesssim \|u\|_1$, \cite{Brezis}.

Additionally, we define the mappings $s:L^2\to H^1$ and $z:H^1\to H^1$ such that for $u\in L^2$ and $v\in H^1$
\begin{equation}
B_b(s[u],w)=-\tfrac{1}{2}(u,w)\quad \text{and}\quad B_b(z[v],w)=-B_c(v,w)\quad\text{for all $w\in H^1$}\ ,
\end{equation}
For $u\in L^2$ we have from the theory of elliptic equations that $\|s[u]\|_2\lesssim \|u\|$. For $v\in H^2$ we have $\|z[v]\|_2\lesssim \|v\|_2$. 

Using these mappings, we rewrite the system (\ref{eq:weakform}) as a system of ordinary differential equations 
\begin{equation}\label{eq:ode}
\bv_t=\bF(\bv)\ ,
\end{equation}
where $\bv=(\eta,\phi)^T$ and $\bF:H^2\times H^2\to H^2\times H^2$ such that
\begin{equation}\label{eq:rhs}
\bF(\bv)=(f[(D+\eta)\nabla\phi],~z[\eta]+s[D^2|\nabla\phi|^2])^T\ .
\end{equation}
Note that we assumed that $\phi,\eta\in H^2$ and $D\in W^{2,\infty}$. Thus, by Grisvard's lemma \cite{BH2015,Grisvard,GR1986}, we have that $(D+\eta)\nabla \phi\in H^1$ and $D^2|\nabla\phi|^2\in L^2$. Thus, $\bF$ is well-defined. Furthermore, $\bF\in C^1(H^2\times H^2,H^2\times H^2)$ with derivative
\begin{equation}
\bF'[\eta^\ast,\phi^\ast](\eta,\phi)=(f[(D+\eta^\ast)\nabla\phi+\eta\nabla\phi^\ast],z[\eta]+s[2D^2\nabla\phi^\ast\nabla\phi])^T\ ,
\end{equation}
and
$$
\begin{aligned}
    \|\bF'[\eta^\ast,\phi^\ast](\eta,\phi)\|_2^2 &= \|f[(D+\eta^\ast)\nabla\phi+\eta\nabla\phi^\ast]\|_2^2+\|z[\eta]+s[2D^2\nabla\phi^\ast\nabla\phi]\|_2^2\\
    &\lesssim \|D\nabla\phi\|_1^2+\|\eta^\ast\nabla\phi\|_1^2+\|\eta\nabla\phi^\ast\|_1^2+\|\eta\|_2^2+\|D^2\nabla\phi^\ast\nabla\phi\|^2\\
    &\leq C(\eta^\ast,\phi^\ast)(\|\eta\|_2^2+\|\phi\|_2^2)\ .
\end{aligned}
$$
From the theory of ordinary differential equations in Banach spaces, we conclude that there is a unique solution $\bv\in H^2\times H^2$ of the problem (\ref{eq:weakform}).
\end{proof}

An immediate consequence of the last result is the following:
\begin{corollary}\label{cor:exist}
There is $T>0$ such that the initial-boundary value problem (\ref{eq:BS})--(\ref{eq:ics}) of the Bona-Smith system with $u_0(\bx)=\nabla\phi_0(\bx)$ has a solution $(\eta,\bu)=(\eta,\nabla\phi)\in C^1([0,T];H^2\times \bH^1)$, where $\bH^1=H^1\times H^1$ and $\phi_0(\bx)=\phi(x,0)$. 
\end{corollary}

\subsection{Uniqueness}

It remains to show that the solution is unique, and this is the focus of the following theorem.

\begin{theorem}\label{thm:existence}
Let $(\eta_0,\bu_0)\in H^2\times H^1$ such that $\bu_0=\nabla\phi_0$ for some $\phi_0\in H^2$. Then, there is a $T>0$ such that the initial-boundary value problem (\ref{eq:BS})--(\ref{eq:ics}) of the Bona-Smith system has a unique solution $(\eta,\bu)=(\eta,\nabla\phi)\in C^1([0,T];H^2\times \bH^1)$ where $\phi_0(\bx)=\phi(\bx,0)$.
\end{theorem}
\begin{proof}
The existence of $\bu$ was established in Corollary \ref{cor:exist}.
To prove the uniqueness of $\bu$ we use contradiction. We assume that there are two solutions $(\eta_1,\bu_1)$ and $(\eta_2,\bu_2)$ satisfying the initial-boundary value problem (\ref{eq:BS})--(\ref{eq:ics}). We define $H=\eta_1-\eta_2$ and $\bU=\bu_1-\bu_2$. We also define for $v\in H^2$ and $\bv\in \bH^1$ the norms 
$$\norma{v}_2^2=\|v\|^2+(b+c)\|D\nabla v\|^2+bc\|\Div\{D^2\nabla v\}\|^2\ ,$$
and 
$$\norma{\bv}_1^2=\|D\bv\|^2+b\|\Div(D^2\bv)\|^2\ .$$
It is easy to see that the two norms $\norma{\cdot}_2$ and $\norma{\cdot}_1$ are equivalent to the corresponding classical norms $\|\cdot\|_2$ and $\|\cdot\|_1$ due to the assumptions on bottom topography (\ref{eq:depthcondition}). Subtracting the corresponding systems (\ref{eq:BS}) for the two solutions $(\eta_1,\bu_1)$ and $(\eta_2,\bu_2)$ we obtain
\begin{align}
&H_t+\Div (D\bU)+\Div(H\bu_1+\eta_2\bU)-b\Div\{D^2\nabla H_t\}=0\ ,\label{eq:sys1}\\
&\bU_t+\nabla H+\tfrac{1}{2}\nabla[\bU\cdot (\bu_1+\bu_2)]-c\nabla\Div(D^2\nabla H)-b\nabla\Div (D^2\bU_t)=0\ .\label{eq:sys2}
\end{align}
Multiplying (\ref{eq:sys1}) by $H-c\Div\,\{D^2\nabla H\}$ and integrating over $\Omega$, we obtain
$$
\begin{aligned}
\tfrac{1}{2}\frac{d}{dt}\norma{H}_2^2&=\int_{\Omega} \bU\cdot D\nabla H +c\Div (D\bU)\Div\{D^2\nabla H\}-\Div(H\bu_1+\eta_2\bU)(H-c\Div\{D^2\nabla H\})~d\bx\\
&\lesssim \|\bU\|\|D\nabla H\|+\|\Div (D\bU)\|\norma{H}_2+(\|\Div(H\bu_1)\|+\|\Div(\eta_2\bU)\|)\norma{H}_2\\
&\lesssim \norma{\bU}_1\norma{H}_2+(\norma{H}_2+\norma{\bU}_1)\norma{H}_2\\
&\lesssim \norma{\bU}_1^2+\norma{H}_2^2\ ,
\end{aligned}$$
where we made use of Grisvard's lemma \cite{Grisvard,GR1986} and the bounds (\ref{eq:depthcondition}) of the bottom topography. Similarly, multiplying equation (\ref{eq:sys2}) by $D^2\bU$ and integrating, we obtain
$$\begin{aligned}
\tfrac{1}{2}\frac{d}{dt}\norma{\bU}_1^2
&=\int_\Omega-\nabla H\cdot (D^2 \bU)+\tfrac{1}{2}\bU\cdot(\bu_1+\bu_1)\Div(D^2\bU)-c\Div(D^2\nabla H)\Div(D^2\bU)~d\bx\\
&\lesssim \|H\|_1\|\bU\|+\|\bU\|\|D^2\bU\|+\|\Div(D^2\nabla H)\|\|\Div(D^2\bU)\|\\
&\lesssim \norma{H}_2\norma{\bU}_1+\norma{\bU}_1^2+\norma{H}_2\norma{\bU}_1\\
&\lesssim \norma{\bU}_1^2+\norma{H}_2^2\ .
\end{aligned}$$
Adding the last two inequalities, yields
$$
\frac{d}{dt}[\norma{H}_2^2+\norma{\bU}_1^2]\lesssim \norma{H}_2^2+\norma{\bU}_1^2\ .$$
Gronwall's inequality yields that $H\equiv 0$ and $\bU\equiv \bo$, which completes the proof.
\end{proof}

\subsection{Equivalence between velocity field and potential formulations}

Problems (\ref{eq:BS}) and (\ref{eq:BS2}) are practically equivalent. Suppose that we have a solution $(\eta,\bu=\nabla\phi)$ of problem (\ref{eq:BS}) and $\phi_0(\bx)=\phi(\bx,0)$. The velocity potential of problem (\ref{eq:BS}) is not unique, in the sense that any function $\phi(\bx,t)+C$ is also a velocity potential of $\bu$. However, fixing the constant $C$, the existence and uniqueness of solutions of (\ref{eq:BS2}) have been established in Theorem \ref{prop:existence}, and the maximal time $T>0$ of the solutions of (\ref{eq:BS}) it will coincide with that of (\ref{eq:BS2}).

Additionally, the solutions to problem (\ref{eq:BS2}) satisfy the energy conservation 
\begin{equation}\label{eq:energymod}
   \tfrac{1}{2}\frac{d}{dt} \int_\Omega [g\eta^2+(D+\eta)|\nabla\phi|^2+cgD^2|\nabla\eta|^2]~d\bx = 0\ .
\end{equation}
This conservation law coincides with the conservation of energy (\ref{eq:energy}) of the original Bona-Smith system (\ref{eq:BS}). To derive it from (\ref{eq:BS2}), we first express system (\ref{eq:BS2}) as:
\begin{equation}\label{eq:intermsys}
\eta_t+\Div R=0, \quad \phi_t+Q=0\ ,
\end{equation}
where $R=(D+\eta)\nabla\phi-bD^2\nabla\eta_t$ and $Q=g\eta+\frac{1}{2}|\nabla\phi|^2-cg\nabla \cdot (D^2\nabla\eta)-b\nabla \cdot (D^2\nabla\phi_t)$. Then, we multiply the first equation of (\ref{eq:intermsys}) with $Q$ and the second with $-\nabla \cdot R$. Equation (\ref{eq:energymod}) follows after integrating over the domain $\Omega$, applying Green's theorem appropriately and simplifying. 

These observations and the fact that the solution $(\eta, \phi)$ of the initial-boundary value problem (\ref{eq:BS2}) with initial data $\nabla \phi(\bx,0)=\bu(\bx,0)$ formulates the solution $(\eta, \bu=\nabla\phi)$ of the corresponding problem (\ref{eq:BS})--(\ref{eq:bcs}) suggests a numerical method that does not require the application of Nitsche's method \cite{Nitche} for approximating slip-wall boundary conditions \cite{KMS2020}. Specifically, we will consider a modified Galerkin / finite element method extending the method proposed in \cite{DMS2010} to preserve the mass, vorticity and total energy of problem (\ref{eq:BS})--(\ref{eq:bcs}) by solving the corresponding problem (\ref{eq:BS2}) with non-essential Neumann boundary conditions. Non-essential boundary conditions are not incorporated in the finite element spaces and are implicitly implied by the weak formulation of the problem \cite{gatica}.

It should be noted that the velocity potential formulation is applicable to other systems within the general class (\ref{eq:abcd})--(\ref{eq:coeffs}), including Peregrine's system. This approach has been previously applied in \cite{Wu1981,WWY1992} to systems similar to Peregrine's using finite difference methods. However, as previously mentioned, solutions to Peregrine's system do not conserve an energy functional, which gives the current study an advantage.

\begin{remark}\label{rmk:smoothness}
The analysis from the previous sections can be applied to the limiting case of the BBM-BBM system with $\theta^2=2/3$ and $c=0$. The primary difference is that, due to the absence of the third derivative term in the system, a unique solution $(\eta,\bu)\in H^1\times \bH^1$ is guaranteed, as demonstrated in \cite{IKKM2021}. It is worth mentioning that the increased regularity of the solutions of the Bona-Smith system with $c>0$ in domain with smooth boundary implies that the solution at the boundary will belong to $H^{1+\frac{1}{2}}(\partial\Omega)$, \cite{GR1986}. 
\end{remark}

\subsection{Line solitary waves}

Line solitary waves are traveling wave solutions that propagate in one direction, denoted by $\balpha$, along which they vary and exhibit symmetry, while remaining constant in the transverse direction $\balpha^\perp$. Here we consider only classical solitary waves that are decreasing monotonically to zero at infinity along the direction of motion. Typically, $\balpha$ represents the direction of a channel through which the solitary wave is propagating and in which it vanishes at large distances from the wave's peak. It is established that when the bottom is flat, with depth $D=D_0$, the Bona-Smith systems have possess solitary waves \cite{Chen1998,Chen2000} as traveling wave solutions. Specifically, contrary to the Euler equations, for any speed $c_s>\sqrt{gD_0}$, there exists a unique solitary wave traveling at that particular speed \cite{DM2008,DM2004}. Moreover, these waves have been experimentally tested and, contrary to the solitary waves of the Euler equations \cite{Tanaka1986,DC2014}, are reportedly stable under both small and large perturbations for all values of speed $c_s>\sqrt{gD_0}$, \cite{DDMM2007}. 

Assume that the domain $\Omega$ is the whole $\mathbb{R}^2$ or a very long channel in the direction $\balpha=(\alpha_x,\alpha_y)^T$ with $|\balpha|=1$ and that the bottom is flat with $D(\bx)=D_0$. Traveling wave solutions propagating along the direction of $\balpha$ with constant speed $c_s>\sqrt{gD_0}$ can be expressed as functions of the form
\begin{equation}\label{eq:ansatz}
\eta(\bx,t)=\eta(\bxi),\qquad \bu(\bx,t)=\bu(\bxi)\ ,
\end{equation}
where $\bxi=(\xi,\zeta)$ with $\xi=\balpha\cdot\bx-c_s t$ and $\zeta=\balpha^{\perp}\cdot\bx$. Here, $\balpha^\perp=(\alpha_y,-\alpha_x)^T$.  

Substituting the {\em ansatz} (\ref{eq:ansatz}) into the Bona-Smith system (\ref{eq:BS}), we obtain the partial differential equations
\begin{equation}\label{eq:petv1}
\begin{aligned}
& -c_s\eta_\xi +\nabla_{\xi\zeta}\cdot[(D_0+\eta)\bA\bu]+c_sbD_0^2\Delta_{\xi\zeta}\eta_\xi=0\ ,\\
& -c_s\bu_\xi+g\bA\nabla_{\xi\zeta}\eta+\tfrac{1}{2}\bA\nabla_{\xi\zeta}|\bu|^2-cgD_0^2\bA\nabla_{\xi\zeta}\Delta_{\xi\zeta}\eta+c_sbD_0^2\bA\nabla_{\xi\zeta}(\nabla_{\xi\zeta}\cdot \bA \bu_\xi)= 0\ ,
\end{aligned}
\end{equation}
where $$\bA=\begin{pmatrix}
    \alpha_x & \alpha_y\\
    \alpha_y & -\alpha_x
\end{pmatrix}\ ,$$
$\nabla_{\xi\zeta}=(\partial_\xi,\partial_\zeta)^T$ and $\Delta_{\xi\zeta}=\nabla_{\xi\zeta}\cdot\nabla_{\xi\zeta}$.
Multiplying the second equation of (\ref{eq:petv1}) by $\bA$, setting $\bw=(w,\tilde{w})^T=\bA\bu$, and using the fact $\bA^2=\bI$, then system (\ref{eq:petv1}) is simplified to
\begin{equation}\label{eq:petv2}
\begin{aligned}
& -c_s\eta_\xi +\nabla_{\xi\zeta}\cdot[(D_0+\eta)\bw]+c_sbD_0^2\Delta_{\xi\zeta}\eta_\xi=0\ ,\\
& -c_s\bw_\xi+g\nabla_{\xi\zeta}\eta+\tfrac{1}{2}\nabla_{\xi\zeta}|\bw|^2-cgD_0^2\nabla_{\xi\zeta}\Delta_{\xi\zeta}\eta+c_sbD_0^2\nabla_{\xi\zeta}(\nabla_{\xi\zeta}\cdot \bw_\xi)= 0\ .
\end{aligned}
\end{equation}
Since we are searching for line solitary waves, we assume that the solution $\bu$ is constant along the direction of $\balpha^{\perp}$. Thus, $\partial_\zeta \eta=\partial_\zeta w=\partial_\zeta\tilde{w}=0$. Integrating  equations (\ref{eq:petv2}) with respect to $\xi$ and assuming that $\lim_{|\xi|\to\infty}w=0$, yields the  differential equations
\begin{equation}\label{eq:odepet}
\begin{aligned}
-c_s\eta +(D_0+\eta)w+c_sbD_0^2\Delta_{\xi\zeta}\eta&=0\ ,\\
-c_s w+g\eta+\tfrac{1}{2}w^2-cgD_0^2\Delta_{\xi\zeta}\eta+c_sbD_0^2 w_{\xi\xi}&= 0\ ,\\
\tilde{w}&=0\ .
\end{aligned}
\end{equation}
This particular system has been analyzed in \cite{Chen1998,Chen2000} and has classical, line solitary waves as traveling wave solutions. In the case where $w=B\eta$, one can find for each value of $\theta^2\in (1/3,1)$ an explicit solution for a solitary wave \cite{Chen2000}, namely
\begin{equation}
\eta(\bxi)=A \sech^2\left(\lambda (\xi+\xi_0)\right), \quad w(\bxi)=B\eta(\bxi), \quad \tilde{w}(\bxi)=0\ ,
\end{equation}
where
\begin{equation}
A = 3\frac{g-D_0B^2}{B^2},\ \lambda=\tfrac{1}{2}\sqrt{\frac{2(g-D_0B^2)}{bD_0^2(2g-D_0B^2)}},\
 c_s=\frac{4 \sqrt{gD_0}\left(\theta^2-\frac{2}{3}\right)}{\sqrt{2\left(\theta^2-\frac{1}{3}\right)\left(1-\theta^2\right)}},\ B=\sqrt{\frac{2g}{D_0}\cdot\frac{1-\theta^2}{\theta^2-\frac{1}{3}}}\ .
\end{equation}
Returning to the original variables, for $1/3<\theta^2<1$, the solitary wave formulas become
\begin{equation}\label{eq:solitwave1}
\eta(\bxi)=A \sech^2\left(\lambda (\balpha\cdot\bx-c_s t)\right), \quad \bu(\bxi)=\balpha B\eta(\bxi)\ ,
\end{equation}
where
\begin{equation}\label{eq:solitwave2}
\begin{aligned}
& A = \frac{9D_0}{2}\cdot\frac{\theta^2-7/9}{1-\theta^2}, &\lambda=\frac{1}{2}\sqrt{\frac{3(\theta^2-7/9)}{D_0(\theta^2-2/3)(\theta^2-1/3)}},\\
& c_s=\frac{4 \sqrt{gD_0}\left(\theta^2-\frac{2}{3}\right)}{\sqrt{2\left(\theta^2-\frac{1}{3}\right)\left(1-\theta^2\right)}}, & B=\sqrt{\frac{2g}{D_0}\cdot\frac{1-\theta^2}{\theta^2-\frac{1}{3}}}\ .
\end{aligned}
\end{equation}

Note that these formulas correspond to a unique solitary wave for each $1/3<\theta^2<1$. Moreover, Bona-Smith systems have also cnoidal wave solutions as periodic traveling waves with similar formulas \cite{ADM2010ii}. These analytical solutions are useful for testing purposes. However, when other solitary waves are involved in applications, we resort to numerical methods for solving system (\ref{eq:odepet}). Such numerical methods can be Galerkin or pseudospectral methods accompanied by Newton's or Petviashvili's method. The Petviashvili method is a modified fixed-point method especially used for the computation of traveling waves \cite{Petv1976}. The Petviashvili method applied to the BBM-BBM system was presented in \cite{KMS2020}. In this work we extend this method for the general case of the Bona-Smith systems at hand and we present it in Section \ref{sec:modgal}. 
 
\subsection{Dispersion relation}\label{sec:disprel}

The Bona-Smith systems (\ref{eq:BS})--(\ref{eq:coeffs}) exhibit dispersion characteristics. For flat bottom topography $D=D_0$, the (scaled) speed of propagation of linear waves governed by the Bona-Smith systems, as a function of the wave number, is expressed through the dispersion relation \cite{IKM2023}:
\begin{equation}\label{eq:disprel}
\frac{c_s}{\sqrt{gD_0}}=\sqrt{\frac{1+c(D_0k)^2}{(1+b(D_0k)^2)^2}}\ ,
\end{equation}
whereas for the Euler equations, it is represented as:
\begin{equation}\label{eq:disprel2}
\frac{c_s}{\sqrt{gD_0}}=\sqrt{\frac{\tanh(D_0k)}{D_0k}}\ .
\end{equation}
\begin{figure}[ht!]
\centering
\includegraphics[width=\columnwidth]{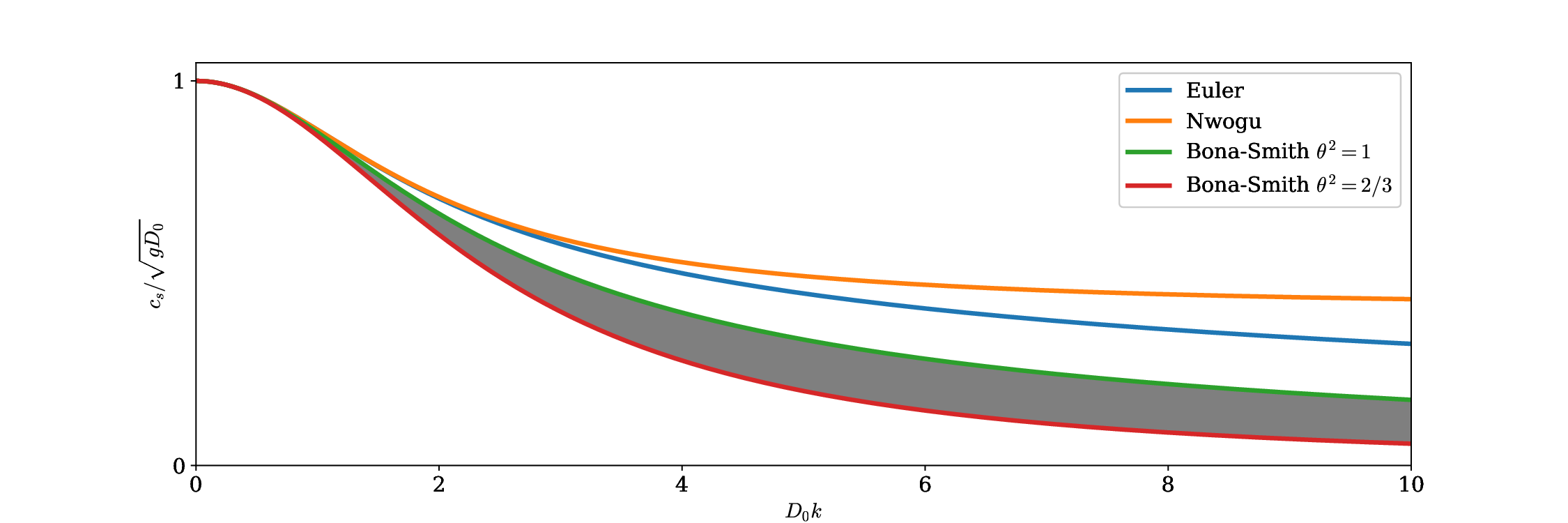}
\caption{Comparison of linear dispersion relations for Bona-Smith systems $2/3\leq \theta^2\leq 1$}
\label{fig:disprel}
\end{figure}

Figure \ref{fig:disprel} provides a graphical comparison between the dispersion relations of the various Bona-Smith systems and the corresponding relation of the Euler equations as well as Nwogu's system which is the optimal one. It is noteworthy that the dispersion relation closest to that of the Euler equations is exhibited by the Bona-Smith system with $\theta^2=1$, and this curve coincides with that of Peregrine's system \cite{P1967} and Serre-Green-Naghdi equations \cite{S1953I}. The Bona-Smith system with $\theta^2=2/3$ aligns with the BBM-BBM (regularized shallow water waves) system. The shaded area corresponds to the dispersion relations of the Bona-Smith systems for $\theta^2\in (2/3,1)$. Note that Nwogu's system has better dispersion characteristics and is optimal compared to other Boussinesq systems of the form (\ref{eq:abcd}).

Evidently, the linear dispersion relation of the Bona-Smith systems closely approximates the corresponding relation of the Euler equations in the long-wave limit. However, it is anticipated that short waves of the Bona-Smith system will deviate from those of the Euler equations. For such waves, Nwogu's system \cite{Nwogu93} emerges as the optimal $abcd$-Boussinesq system. Although it belongs to a distinct category, it remains unknown whether its solutions preserve any reasonable form of total energy.

\section{Energy-preserving numerical method}\label{sec:modgal}

Consider the initial-boundary value problem 
\begin{equation}\label{eq:BS22}
\begin{aligned}
&\begin{aligned}
&\eta_t+\Div [(D+\eta)\nabla\phi]-b\Div\left\{D^2\nabla\eta_t\right\}=0\ ,\\
&\phi_t+g\eta+\tfrac{1}{2}|\nabla\phi|^2-cg\Div(D^2\nabla\eta)-b\Div(D^2\nabla\phi_t)=0\ ,
\end{aligned}\quad \text{in $\Omega$}\ ,\\
&\eta(\bx,0)=\eta_0(\bx)\quad \text{and}\quad \phi(\bx,0)=\phi_0(\bx)\ ,\\
&\nabla\eta\cdot\bn=0\quad \text{and}\quad \nabla\phi\cdot\bn=0\quad \text{on $\partial\Omega$}\ .
\end{aligned}
\end{equation}
Let $\Omega$ be a polygonal domain and $b,d>0$, $c\geq 0$ constants such that the problem (\ref{eq:BS2}) has a unique smooth solution $(\eta,\bu)$. Upon obtaining an approximate solution $(\teta,\tphi)$ to (\ref{eq:BS22}), we establish the corresponding approximation of (\ref{eq:BS}) by defining $(\teta,\tbu)=(\teta,\nabla\tphi)$. In what follows we present a conservative numerical method for the problem (\ref{eq:BS22}).

\subsection{Energy-preserving modified Galerkin method}\label{sec:enepmg}

Assume that $\mathcal{T}_h$ is a regular triangulation of $\Omega$ with triangles $\tau$ of maximum edge $h$. We define the space of Lagrange finite elements  of degree $r$
$$\mathcal{P}_h^r=\{f\in C(\Omega)~ :~ f|_\tau \in \mathbb{P}_r(\tau) \text{ for all } \tau\in \mathcal{T}_h \}\ ,$$
where $\mathbb{P}_r$ denotes the space of polynomials of degree $r$. The space $\mathcal{P}^r_h$ is a subset of $C(\bar{\Omega})\cap H^1(\Omega)$ \cite{EG2010}. Furthermore, we denote by $P^r$ the $L^2$-projection of any function $f\in L^2(\Omega)$ onto the space $\mathcal{P}^r_h$ defined by
$$(P^r[f],\chi)=(f,\chi)\quad \text{ for all } \chi \in \mathcal{P}^r_h\ .$$

If $w,v\in H^1$, we define the modified discrete Laplacian operator $\Delta_h:H^1\to \mathcal{P}^r_h$ such that 
$$(\Delta_h [w;v],\chi)=-(w\nabla v,\nabla \chi)\qquad \text{ for all $\chi\in \mathcal{P}^r_h$}\ .$$

Then, we define the numerical solution of the problem (\ref{eq:BS22}) as $\teta,\tphi\in\mathcal{P}^r_h$ such that
\begin{equation}\label{eq:semid1}
\begin{aligned}
& (\teta_t,\chi)+(\Delta_h[D+\teta;\tphi],\chi)-b(\Delta_h[D^2; \teta_t], \chi)=0 & \text{for all $\chi\in \mathcal{P}^r_h$}\ ,\\
& (\tphi_t,\psi)+g(\teta,\psi)+\tfrac{1}{2}(P^r[|\nabla\tphi|^2],\psi)-cg(\Delta_h[D^2;\teta],\psi)-b(\Delta_h[D^2;\tphi_t],\psi)=0 & \text{for all $\psi\in\mathcal{P}^r_h$}\ ,
\end{aligned}
\end{equation}
along with the initial conditions
\begin{equation}\label{eq:incon}
\teta(\bx,0)=P^r[\eta_0(\bx)],\quad \tphi(\bx,0)=P^r[\phi_0(\bx)]\ .
\end{equation}
The semidiscretization (\ref{eq:semid1}) is an energy-conservative generalization of the modified Galerkin method introduced in \cite{DMS2010}. It also extends the method presented in \cite{MRKS2021} to one-dimensional systems with float bottom and with high-order derivatives, similar to those considered here.

We assume that the initial-value problem (\ref{eq:semid1}), (\ref{eq:incon}) has a unique solution in the interval $t\in[0,T]$ for some $T>0$. Then, the corresponding solution $(\teta,\tilde{\phi})$ of the modified Galerkin semidiscretization (\ref{eq:semid1})--(\ref{eq:incon}) preserves the mass 
$$\mathcal{M}(t;\teta)=\int_{\Omega}\teta(\bx,t)~dt\ ,$$
and the energy
$$\mathcal{E}(t;\teta,\tphi)=\tfrac{1}{2} \int_\Omega [g\teta^2+(D+\teta)|\nabla\tphi|^2+cgD^2|\nabla\teta|^2]~d\bx \ ,$$
in the sense $\mathcal{M}(t;\teta)=\mathcal{M}(0;\teta)$, $\mathcal{E}(t;\teta,\tphi)=\mathcal{E}(0;\teta,\tphi)$. Specifically, we have the following:

\begin{proposition}
The solution $(\teta,\tilde{\phi})$ of the initial-value problem (\ref{eq:semid1}), (\ref{eq:incon}) satisfies the following conservation laws for all $t\in [0,T]$:
\begin{enumerate}
\item[(i)] Conservation of mass:
$$\frac{d}{dt}\mathcal{M}(t;\teta)=0\ .$$
\item[(ii)] Conservation of energy:
$$\frac{d}{dt}\mathcal{E}(t;\teta,\tphi)=0\ .$$
\end{enumerate}
\end{proposition}
\begin{proof}
The mass conservation follows by taking $\chi=1$ in the first equation of (\ref{eq:semid1}). To prove (ii), we define
$$\tilde{R}=\Delta_h[D+\teta;\tphi]-b\Delta_h[D^2; \teta_t]\in \mathcal{P}^r_h \quad \text{ and } \quad
\tilde{Q}=g\teta+\tfrac{1}{2}P^r[|\nabla\phi|^2]-cg\Delta_h[D^2;\teta]-b\Delta_h[D^2;\tphi_t]\in \mathcal{P}^r_h\ .$$
Then, the system (\ref{eq:semid1}) can be expressed as 

\begin{eqnarray*}
(\teta_t,\chi)+(\tilde{R}, \chi)=0 & \text{ for all $\chi\in \mathcal{P}^r_h$}\ ,\\
(\tphi_t,\psi)+(\tilde{Q}, \psi)=0 & \text{ for all $\psi\in\mathcal{P}^r_h$}\ .
\end{eqnarray*}

By taking $\chi=\tilde{Q}$ and $\psi=\tilde{R}$ and subtracting the last two equations, yields $(\teta_t,\tilde{Q})-(\tphi_t,\tilde{R})=0$. On the other hand, we have
$$
\begin{aligned}
(\teta_t,\tilde{Q})-(\tphi_t,\tilde{R})&=g(\teta_t,\teta)+\tfrac{1}{2}(\teta_t,P^r[|\nabla\phi|^2])-cg(\teta_t,\Delta_h[D^2;\teta])-b(\teta_t,\Delta_h[D^2;\tphi_t])\\
&\qquad -(\tphi_t,\Delta_h[D+\teta;\tphi])+b(\tphi_t,\Delta_h[D^2; \teta_t])\\
&=g(\teta_t,\teta)+\tfrac{1}{2}(\teta_t,|\nabla\phi|^2)+cg(\nabla\teta_t,D^2\nabla\teta)+b(\nabla\teta_t,D^2\nabla\tphi_t)\\
&\qquad +(\nabla\tphi_t,(D+\teta)\nabla\tphi])-b(\nabla\tphi_t,D^2 \nabla\teta_t)\\
&=\tfrac{1}{2}\frac{d}{dt}\int_\Omega [g\teta^2+cgD^2|\nabla\teta|^2+(D+\teta)|\nabla\tphi|^2~d\bx \ .
\end{aligned}
$$
Thus, $\frac{d}{dt}\mathcal{E}(t;\teta,\tphi)=0$.
\end{proof}

By introducing auxiliary unknown functions $\tv,\tw,\tz,\tf:[0,t]\to \mathcal{P}^r $ for the $L^2$-projection and the various discrete Laplacians appearing in the semidiscretization (\ref{eq:semid1}), we can express system (\ref{eq:semid1}) in the following mixed formulation:
\begin{equation}\label{eq:mixedform}
\begin{aligned}
(\teta_t,\chi)-b(\tv_t,\chi)&=(\tw,\chi) & \text{for all $\chi\in \mathcal{P}^r_h$}\ ,\\
(D^2\nabla\teta_t,\nabla\omega)+(\tv_t,\omega) &=0 & \text{for all $\omega\in \mathcal{P}^r_h$}\ ,\\
(\tw,\zeta)&=((D+\teta)\nabla\tphi,\nabla\zeta) & \text{for all $\zeta\in \mathcal{P}^r_h$}\ ,\\
(\tphi_t,\psi)-b(\tz_t,\psi)&=cg(\tv,\psi)-g(\teta,\psi)-\tfrac{1}{2}(\tf,\psi)  & \text{for all $\psi\in \mathcal{P}^r_h$}\ ,\\
(D^2\nabla\tphi_t,\nabla\xi)+(\tz_t,\xi)&=0  & \text{for all $\xi\in \mathcal{P}^r_h$}\ ,\\
(\tf,\theta)&=(|\nabla\tphi|^2,\theta)  & \text{for all $\theta\in \mathcal{P}^r_h$}\ ,\\
\end{aligned}
\end{equation}
along with the initial conditions:
\begin{equation}\label{eq:initcondmix}
\begin{aligned}
&\teta(\bx,0)=P^r[\eta_0(\bx)],\qquad \tv(\bx,0)=\Delta_h[D^2;\eta_0](\bx) \ ,\\
&\tphi(\bx,0)=P^r[\phi_0(\bx)],\qquad \tz(\bz,0)=\Delta_h[D^2;\phi_0](\bx) \ .
\end{aligned}
\end{equation}
This mixed formulation is convenient for implementation purposes.

It is worth noting that since the numerical velocity field $\tbu$ is produced by the potential $\tphi$ as $\tbu=\nabla\tphi$, it will satisfy
$\Curl \tilde{\bu} =\Curl \nabla\tphi =0$ for all $\bx\in\Omega$ and $t>0$. Therefore, in addition to mass and energy, the proposed numerical method guarantees the conservation of vorticity as well.

\subsection{Time discretization}

For the time discretization of the autonomous system of ordinary differential equations (\ref{eq:mixedform}), we can employ an explicit relaxation Runge-Kutta method \cite{KE2019,RK2020}. As we shall see in Section \ref{sec:expval} experimentally, the system (\ref{eq:mixedform}) appears to be non-stiff, and explicit Runge-Kutta methods should be efficient. It is noteworthy that for Bona-Smith type systems containing similar regularization operators (BBM-terms), it was proven that there is no timestep restriction requirement for the standard Galerkin semidiscretization \cite{DMS2010}. However, for the current, conservative Galerkin semidiscretization further analysis is required to study the effects of the regularization terms, and for this reason we will use sufficiently small timestep $\Delta t<\Delta x$.

Consider the initial-value problem
\begin{equation}\label{eq:ode2}
\begin{aligned}
    \frac{d}{dt}\by(t)&=\btf(t,\by(t)),\quad t\in [0,T],\\
    \by(0)&=\by_0\ ,
\end{aligned}
\end{equation}
where $\by(t)=[\by_i(t)]$ is the vector of unknown functions. Let $\Delta t$ be a fixed timestep such that $0<\Delta t<1$. We define $J=\{t_0,t_1,\dots, t_K\}$ a uniform grid of the interval $[0,T]$ such that $t_{i+1}=t_i+\Delta t$ with $t_0=0$ and $t_K=T$. An explicit Runge-Kutta method with $s$ intermediate stages can be fully described by its Butcher tableau
\begin{align}\label{eq:butcher}
\begin{array}{c|c}
c & A\\
\hline
& b^\mathrm{T}
\end{array},
\end{align}
where $A=[a_{ij}]_{i,j=1}^s$ is an $s\times s$ lower-triangular matrix with zeros
in the main diagonal, and $b=[b_j]_{j=1}^s$ and $c=[c_j]_{j=1}^s$ are
$s$-dimensional vectors. If $\bgamma^n=\{\gamma^1,\gamma^2,\dots,\gamma^K\}$ is a vector containing the so-called relaxation parameters $\gamma^n\in (0,1]$, and $\by^n$ is the approximation of $\by(t_{n-1}+\gamma^n\Delta t)$, then a step of an explicit relaxation Runge-Kutta method corresponding to the Butcher tableau (\ref{eq:butcher})
can be expressed as
\begin{align}
\tilde{\by}^i&=\by^n+\Delta t\sum_{j=1}^{i-1} a_{ij}\,\btf(t_n+c_j\Delta t, \tilde{\by}^j), \quad i=1,2,\ldots,s, \label{eq:stages}\\
\by^{n+1}&=\by^n+\gamma^n\Delta t\sum_{i=1}^s b_i \,\btf(t_n+c_i\Delta t,\tilde{\by}^i). \label{eq:updates}
\end{align}
The parameters $\gamma^n$ are typically close to $1$ and can be determined such as the energy of the system is preserved. For example, if the energy of the system (\ref{eq:ode2}) is $\mathcal{E}(t;\by)=\mathcal{E}(0;\by)$, then we can choose the parameter $\gamma^n$ such that 
$$\mathcal{E}(t_{n+1}; \by^{n+1})=\mathcal{E}(t_n;\by^n)\ .$$

For simplicity, we will denote the sum in the right-hand side of (\ref{eq:updates}) by
$$\bd^n=\sum_{i=1}^s b_i \,\btf(t_n+c_i\Delta t,\tilde{\by}^i)\ .$$

In our case, we define the vector $\by^n=(\mathcal{H}^n,\mathcal{W}^n,\Phi^n,\mathcal{Z}^n)^T$, where $\mathcal{H}^n, \mathcal{W}^n, \Phi^n, \mathcal{Z}^n\in \mathcal{P}^r_h$ are the fully discrete approximations of $\tilde{\eta}(\cdot,t_n)$, $\tilde{w}(\cdot,t_n)$, $\tilde{\phi}(\cdot,t_n)$ and $\tilde{z}(\cdot,t_n)$, respectively.  Then, the updates for $\mathcal{H}^{n}$ and $\Phi^n$ of (\ref{eq:updates}) can be expressed in the form
\begin{equation}\label{eq:updates2}
\begin{aligned}
\mathcal{H}^{n+1}&=\mathcal{H}^n+\gamma^n\Delta t d_\eta^n, \\
\Phi^{n+1}&=\Phi^n+\gamma^n\Delta t d_\phi^n,
\end{aligned}
\qquad\text{for $n=0,1,2,\dots$}\ ,
\end{equation}
where $d_\eta^n$ and $d_\phi^n$ are the components of the vector $\bd^n$ corresponding to the solution $\mathcal{H}^n$ and $\Phi^n$, respectively. To find the relaxation parameter such as the energy between times $t_n$ and $t_{n+1}$ remain constant, we need to solve the equation
\begin{equation}\label{eq:discenergy}
\mathcal{E}(t_{n+1};\mathcal{H}^{n+1},\Phi^{n+1})=\mathcal{E}(t_n,\mathcal{H}^n,\Phi^n)\ .
\end{equation}
Substituting the updates (\ref{eq:updates2}) into equation (\ref{eq:discenergy}), yields the cubic equation
\begin{equation}\label{eq:cubiceq}
\Alpha x^3+\Beta x^2+\Gamma x=0\ ,
\end{equation}
where $x=\gamma^n\Delta t$ and 
\begin{equation}
\begin{aligned}
\Alpha &= \int_\Omega d_\eta|\nabla d_\phi|^2~d\bx\ ,\\
\Beta &= \int_\Omega gd_\eta^2+cgD^2|\nabla d_\eta|^2+(D+\mathcal{H}^n)|\nabla d_\phi|^2+2d_\eta \nabla\Phi^n\cdot \nabla d_\phi~d\bx\ ,\\
\Gamma &= \int_\Omega (2 g \mathcal{H}^n+|\nabla\Phi^n|^2) d_\eta+2cgD^2\nabla \mathcal{H}^n\cdot \nabla d_\eta+2(D+\mathcal{H}^n)\nabla\Phi^n\cdot\nabla d_\phi~d\bx\ .
\end{aligned}
\end{equation}
Therefore, there are three roots, namely $\gamma^n_0=0$ and
\begin{equation}
\gamma^n_{\pm}\cdot \Delta t=\frac{-\Beta\pm\sqrt{\Beta^2-4\Alpha\Gamma}}{2\Alpha}\ .
\end{equation}
The parameter $\gamma^n$ depends on the current solution of the problem making the determination of its consistency challenging. For this reason, we cannot study the roots $\gamma^n$ further. However, whenever the explicit Runge-Kutta method converges, let's say at order $\Delta t^p$, it is expected that there is a root $\gamma^n\approx 1$ such that $|\gamma^n-1|=O(\Delta t^{p-1})$ \cite{KE2019}. In practice, after choosing $\Delta t\ll 1$ and excluding the solution $\gamma^n_0$, we solve numerically  the resulting quadratic equation (\ref{eq:cubiceq}) using, for instance, the secant method to avoid catastrophic cancellation phenomena.

For the numerical experiments of this work, we will use the relaxation Runge-Kutta method
related to the classical four-stage, fourth-order Runge-Kutta method given by
the Butcher tableau:
\begin{align}\label{eq:RK4}
\begin{array}{c|c}
c & A\\
\hline
& b^\mathrm{T}
\end{array} ~=~ \begin{array}{c|cccc}
0 &  0 & 0 & 0  & 0  \\
1/2 &  1/2 & 0  & 0 & 0 \\
1/2 & 0 & 1/2 & 0 & 0 \\
1 & 0 & 0 & 1 & 0 \\
\hline
& 1/6 & 1/3 & 1/3 & 1/6
\end{array}.
\end{align}
Note that this numerical method has been successfully applied to the one-dimensional BBM-BBM system, as demonstrated in \cite{MRKS2021, LR2024, RMK}, in conjunction with conservative spatial semidiscretizations. However, this is the first time we are applying the method to two-dimensional Boussinesq systems.

\subsection{Computation of initial velocity potential}

Our objective is to generate approximations to solutions of system (\ref{eq:BS}) with given initial conditions $\eta_0(\bx)$ and $\bu_0(\bx)$ in $\Omega$ using the solutions of (\ref{eq:BS2}). However, system (\ref{eq:BS2}) requires an initial velocity potential $\phi_0$. Determining the initial velocity potential $\phi_0$ from the initial velocity field $\bu_0$ presents a challenge. Initially, we establish an initial velocity potential profile $\phi(\bx,0)=\phi_0(\bx)$ by solving the Poisson equation
$$\Delta \varphi_0=\Div \bu_0\quad\text{in $\Omega$}\ ,$$
with boundary condition $\nabla\varphi_0\cdot\bn=0$ on $\partial\Omega$. We use the letter $\varphi_0$ instead of $\phi_0$ for the solution of the Laplace equation because this problem has a unique solution subject to vertical translations. To demonstrate this, let's consider two solutions $\varphi_0^1$ and $\varphi_0^2$ and define $\psi = \varphi_0^1 - \varphi_0^2$. This function satisfies the homogeneous Laplace equation $\Delta \psi = 0$ in $\Omega$ with boundary condition $\nabla \psi \cdot \mathbf{n} = 0$ on $\partial\Omega$. Multiplying the Laplace equation by $\psi$ and integrating over $\Omega$, we obtain
$$\int_{\Omega}|\nabla\psi|^2~d\bx=0\ .$$ 
Hence, $\nabla \psi = 0$, which is equivalent to $\varphi_0^1 = \varphi_0^2 + C$ for any constant $C$. Nevertheless, this doesn't impose a restriction, as the solutions to the problem (\ref{eq:BS})--(\ref{eq:ics}) are unaffected by vertical translations of the velocity potential of (\ref{eq:BS2}) since $\bu=\nabla\phi$.

We proceed to solve for the initial potential $\varphi_0$ numerically using the standard Galerkin method with the triangular grid $\mathcal{T}_h$ defined in Section \ref{sec:enepmg} for the solution of (\ref{eq:BS2}). We obtain an approximation $\tilde{\varphi}_0 \in \mathcal{P}^r_h$. Conventionally, we select the initial velocity potential function as $\tilde{\phi}_0 = \tilde{\varphi}_0 - \min_{\Omega} \tilde{\varphi}_0$.

\subsection{Generation of solitary waves using the Petviashvili method}\label{sec:petviashvili}

To generate numerical solitary waves that travel with constant speed $c_s$, we employ the Petviashvili method, which is a modified fixed-point iteration applied to the time-independent equations (\ref{eq:odepet}). Specifically, to apply the Petviashvili method, we write equations (\ref{eq:odepet}) in the form
\begin{equation}\label{eq:peta1} 
\mathcal{L}\bw=\mathcal{N}(\bw)\ ,
\end{equation}
where $\bw=(\eta,w,\tw)$ and
\begin{equation}\label{eq:peta2}
\mathcal{L}=\begin{pmatrix}
c_s-b c_s D_0^2\Delta_{\xi\zeta} & -D_0 & 0\\
-g+cgD_0^2\Delta_{\xi\zeta} & c_s-bc_sD_0^2\partial_\xi^2 & 0\\
0 & 0 & 1
\end{pmatrix} \quad \text{and}\quad \mathcal{N}(\bw)=\begin{pmatrix}
\eta w\\
\tfrac{1}{2}w^2\\
0
\end{pmatrix}\ .
\end{equation}
The Galerkin finite element method for solving system (\ref{eq:peta1})--(\ref{eq:peta2}) is the following: seek for approximation $\bw_h=(\eta_h,w_h,\tw_h)^T\in (\mathcal{P}_h^r)^3$ such that
\begin{equation}\label{eq:petfem1}
\mathcal{L}_h(\bw_h,\bchi)=(\mathcal{N}(\bw_h),\bchi),\quad \text{for all}\quad \bchi\in (\mathcal{P}_h^r)^3\ ,
\end{equation}
where $\mathcal{L}_h$ is the bilinear form
\begin{equation}\label{eq:petfem2}
\begin{aligned}
\mathcal{L}_h(\bw,\bchi)= & c_s(\eta,\phi)+bc_s D_0^2(\nabla_{\xi\zeta}\eta,\nabla_{\xi\zeta}\phi)-D_0(w,\phi)\\
&+c_s(w,\chi)+b c_s D_0^2(w_\xi,\chi_\xi)-g(\eta,\chi)-cgD_0^2(\nabla_{\xi\zeta}\eta,\nabla_{\xi\zeta}\chi)+(\tw,\psi)\ ,
\end{aligned}
\end{equation}
for all $\bw=(\eta,w,\tw)^T\in (\mathcal{P}_h^r)^3$ and $\bchi=(\phi,\chi,\psi)^T\in (\mathcal{P}_h^r)^3$. Given an initial guess $\bw_h^0$ for the solitary wave solution, the Petviashvili method for solving the nonlinear system of equations (\ref{eq:petfem1})--(\ref{eq:petfem2}) is the fixed-point iteration
\begin{equation}\label{eq:npetfem3}
\mathcal{L}_h(\bw_h^{(n+1)},\bchi)=m_n^p(\mathcal{N}(\bw_h^n),\bchi),\quad n=0,1,\dots\ ,
\end{equation}
for all $\bchi\in (\mathcal{P}_h^r)^3$, $p$ appropriate constant, and $m_n$ the scalar
$$m_n=\frac{\mathcal{L}_h(\bw_h^n,\bw_h^n)}{(\mathcal{N}(\bw_h^n),\bw_h^n)}\ .$$
The initial guess $\bw_h^0$ can be taken to be the $L^2$-projection of the corresponding solitary wave of the Serre equations \cite{MID2014} for the same speed $c_s$. 

Typically, the Petviashvili iteration will be terminated when the normalized residual is less than a prescribed tolerance $\delta$. Specifically, we will consider
$$|\mathcal{L}_h(\bw_h^n,\bw_h^n)-(\mathcal{N}(\bw_h^n),\bw_h^n)|/\|\bw_h^n\|_2<\delta\ ,$$
where $\delta$ is the desired tolerance. In the experiments of this paper, we took $\delta=10^{-6}$, and the Petviashvili iteration with $p=2$ required less than 10 iterations to converge within the prescribed tolerance.

\section{Experimental validation}\label{sec:expval}

In this section, we present several numerical experiments to provide experimental justification for both the mathematical model and the numerical method. For the numerical implementation of the fully-discrete scheme, we employed the Python library Fenics \cite{fenics2012}. Fenics uses appropriate quadrature methods to evaluate integrals of polynomials exactly (within machine precision) as they appear in (\ref{eq:mixedform}). Unless otherwise specified, all experiments in this study employ the Bona-Smith system with $\theta^2=1$. It is important to note that similar results can be obtained for any value of $\theta$ in the range $[2/3, 1]$. Similar results for $\theta^2 = 2/3$ with non-conservative numerical method have been presented in \cite{KMS2020, IKKM2021}.
 
\subsection{Experimental convergence rates}\label{sec:expconv}

To study the convergence of the conservative numerical method, we consider the initial-boundary value problem (\ref{eq:BS2}) with bottom topography $D(x,y)=-(x+y)/20+3/2$, $g=1$ and $\Omega=[0,1]\times [0,1]$. Moreover, we consider the functions
$$
\eta(\bx,t)=e^t\cos(2\pi x)\cos(2\pi y), \quad \phi(\bx,t)=e^t\cos(\pi x)\cos(\pi y), \quad \bx=(x,y)^T\in\Omega\quad \text{and}\quad t\geq 0\ ,
$$
as exact solution to the corresponding initial-boundary value problem (\ref{eq:BS2}), with appropriate forcing terms on the right-hand side of the system. Although these terms were computed, they are omitted from this text due to their complex form. Then, we solve the system (\ref{eq:mixedform}) for various uniform triangulations and we monitor the $L^2$ and $H^1$ norms of the errors between the numerical and the exact solutions in the spaces $\mathcal{P}^r_h$ for $r=1,2,3,4$. Assuming that the method converges, we define 
$E_i[h,V]=\|\tilde{V}-V\|_i$ for $i=0,1$, where $V$ can be any of $\eta$ or $\phi$. If $h_1$, $h_2$ correspond to the maximum edges of a triangle of two different uniform triangulations with right-angle triangles of $\Omega$, then the experimental order of convergence is defined as
$$r=\log \frac{E_i[h_1,V]}{E_i[h_2,V]}/\log \frac{h_1}{h_2}\ .$$
Because the focus was on the Galerkin semi-discretization, we used small $\Delta t=5\times 10^{-4}$ for all finite element spaces and grids to eliminate the influence of the time discretization to the error. For the computations of the norms, we utilized the corresponding Fenics functions \cite{fenics2012}.

\begin{table}[t]
{\small
\begin{tabular}{ccccccccc}
\hline
 $h$  & $E_0[h,\phi]$ & $r$ & $E_0[h,\eta]$ & $r$ & $E_1[h,\phi]$ & $r$ & $E_1[h,\eta]$ & $r$ \\
 \hline
 $ 1.250\times 10^{-1} $ & $ 4.844\times 10^{-2} $ &  --       & $ 5.364\times 10^{-2} $ &  --       & $ 6.272\times 10^{-1} $ &  --       & $ 2.515\times 10^{0} $ &  --       \\
 $ 8.333\times 10^{-2} $ & $ 2.127\times 10^{-2} $ & $ 2.030 $ & $ 2.356\times 10^{-2} $ & $ 2.029 $ & $ 4.173\times 10^{-1} $ & $ 1.005 $ & $ 1.673\times 10^{ 0} $ & $ 1.006 $ \\
 $ 6.250\times 10^{-2} $ & $ 1.192\times 10^{-2} $ & $ 2.014 $ & $ 1.320\times 10^{-2} $ & $ 2.014 $ & $ 3.128\times 10^{-1} $ & $ 1.002 $ & $ 1.254\times 10^{ 0} $ & $ 1.003 $ \\
 $ 5.000\times 10^{-2} $ & $ 7.613\times 10^{-3} $ & $ 2.008 $ & $ 8.431\times 10^{-3} $ & $ 2.008 $ & $ 2.501\times 10^{-1} $ & $ 1.001 $ & $ 1.003\times 10^{ 0} $ & $ 1.002 $ \\
 $ 4.167\times 10^{-2} $ & $ 5.281\times 10^{-3} $ & $ 2.005 $ & $ 5.849\times 10^{-3} $ & $ 2.006 $ & $ 2.084\times 10^{-1} $ & $ 1.001 $ & $ 8.356\times 10^{-1} $ & $ 1.001 $ \\
 $ 3.571\times 10^{-2} $ & $ 3.878\times 10^{-3} $ & $ 2.004 $ & $ 4.295\times 10^{-3} $ & $ 2.004 $ & $ 1.786\times 10^{-1} $ & $ 1.001 $ & $ 7.162\times 10^{-1} $ & $ 1.001 $ \\
 $ 3.125\times 10^{-2} $ & $ 2.968\times 10^{-3} $ & $ 2.003 $ & $ 3.287\times 10^{-3} $ & $ 2.003 $ & $ 1.563\times 10^{-1} $ & $ 1.000 $ & $ 6.266\times 10^{-1} $ & $ 1.001 $ \\
 \hline
\end{tabular}
}
\caption{Experimental convergence rate $r$ in $L^2$ and $H^1$ norms with $\mathcal{P}^1_h$ elements}\label{tab:eoc1}
\end{table}

\begin{table}[t]
{\small
\begin{tabular}{ccccccccc}
\hline
 $h$  & $E_0[h,\phi]$ & $r$ & $E_0[h,\eta]$ & $r$ & $E_1[h,\phi]$ & $r$ & $E_1[h,\eta]$ & $r$ \\
 \hline
 $ 1.250\times 10^{-1} $ & $ 4.718\times 10^{-4} $ &  --       & $ 3.329\times 10^{-3} $ &  --       & $ 3.191\times 10^{-2} $ &  --       & $ 2.541\times 10^{-1} $ &  --       \\
 $ 8.333\times 10^{-2} $ & $ 1.365\times 10^{-4} $ & $ 3.059 $ & $ 1.027\times 10^{-3} $ & $ 2.901 $ & $ 1.421\times 10^{-2} $ & $ 1.996 $ & $ 1.136\times 10^{-1} $ & $ 1.986 $ \\
 $ 6.250\times 10^{-2} $ & $ 5.705\times 10^{-5} $ & $ 3.032 $ & $ 4.398\times 10^{-4} $ & $ 2.947 $ & $ 7.996\times 10^{-3} $ & $ 1.998 $ & $ 6.397\times 10^{-2} $ & $ 1.995 $ \\
 $ 5.000\times 10^{-2} $ & $ 2.908\times 10^{-5} $ & $ 3.020 $ & $ 2.268\times 10^{-4} $ & $ 2.967 $ & $ 5.119\times 10^{-3} $ & $ 1.999 $ & $ 4.097\times 10^{-2} $ & $ 1.998 $ \\
 $ 4.167\times 10^{-2} $ & $ 1.679\times 10^{-5} $ & $ 3.014 $ & $ 1.318\times 10^{-4} $ & $ 2.978 $ & $ 3.555\times 10^{-3} $ & $ 1.999 $ & $ 2.846\times 10^{-2} $ & $ 1.999 $ \\
 $ 3.571\times 10^{-2} $ & $ 1.055\times 10^{-5} $ & $ 3.010 $ & $ 8.320\times 10^{-5} $ & $ 2.984 $ & $ 2.612\times 10^{-3} $ & $ 2.000 $ & $ 2.091\times 10^{-2} $ & $ 1.999 $ \\
 $ 3.125\times 10^{-2} $ & $ 7.064\times 10^{-6} $ & $ 3.008 $ & $ 5.583\times 10^{-5} $ & $ 2.988 $ & $ 2.000\times 10^{-3} $ & $ 2.000 $ & $ 1.601\times 10^{-2} $ & $ 1.999 $ \\
 \hline
\end{tabular}
}
\caption{Experimental convergence rate $r$ in $L^2$ and $H^1$ norms with $\mathcal{P}^2_h$ elements}\label{tab:eoc2}
\end{table}

\begin{table}[t]
{\small
\begin{tabular}{ccccccccc}
\hline
 $h$  & $E_0[h,\phi]$ & $r$ & $E_0[h,\eta]$ & $r$ & $E_1[h,\phi]$ & $r$ & $E_1[h,\eta]$ & $r$ \\
 \hline
 $ 1.250\times 10^{-1} $ & $ 8.856\times 10^{-6} $ &  --       & $ 1.158\times 10^{-4} $ &  --       & $ 9.763\times 10^{-4} $ &  --       & $ 1.602\times 10^{-2} $ &  --       \\
 $ 8.333\times 10^{-2} $ & $ 1.735\times 10^{-6} $ & $ 4.020 $ & $ 2.255\times 10^{-5} $ & $ 4.034 $ & $ 2.891\times 10^{-4} $ & $ 3.001 $ & $ 4.746\times 10^{-3} $ & $ 3.000 $ \\
 $ 6.250\times 10^{-2} $ & $ 5.475\times 10^{-7} $ & $ 4.009 $ & $ 7.101\times 10^{-6} $ & $ 4.017 $ & $ 1.220\times 10^{-4} $ & $ 3.001 $ & $ 2.002\times 10^{-3} $ & $ 3.000 $ \\
 $ 5.000\times 10^{-2} $ & $ 2.240\times 10^{-7} $ & $ 4.005 $ & $ 2.902\times 10^{-6} $ & $ 4.010 $ & $ 6.244\times 10^{-5} $ & $ 3.000 $ & $ 1.025\times 10^{-3} $ & $ 3.000 $ \\
 $ 4.167\times 10^{-2} $ & $ 1.080\times 10^{-7} $ & $ 4.003 $ & $ 1.398\times 10^{-6} $ & $ 4.007 $ & $ 3.613\times 10^{-5} $ & $ 3.000 $ & $ 5.931\times 10^{-4} $ & $ 3.000 $ \\
 $ 3.571\times 10^{-2} $ & $ 5.826\times 10^{-8} $ & $ 4.002 $ & $ 7.538\times 10^{-7} $ & $ 4.005 $ & $ 2.275\times 10^{-5} $ & $ 3.000 $ & $ 3.735\times 10^{-4} $ & $ 3.000 $ \\
 $ 3.125\times 10^{-2} $ & $ 3.414\times 10^{-8} $ & $ 4.002 $ & $ 4.417\times 10^{-7} $ & $ 4.004 $ & $ 1.524\times 10^{-5} $ & $ 3.000 $ & $ 2.502\times 10^{-4} $ & $ 3.000 $ \\
 \hline
\end{tabular}
}
\caption{Experimental convergence rate $r$ in $L^2$ and $H^1$ norms with $\mathcal{P}^3_h$ elements}\label{tab:eoc3}
\end{table}

\begin{table}[t]
{\small
\begin{tabular}{ccccccccc}
\hline
 $h$  & $E_0[h,\phi]$ & $r$ & $E_0[h,\eta]$ & $r$ & $E_1[h,\phi]$ & $r$ & $E_1[h,\eta]$ & $r$ \\
 \hline
 $ 1.250\times 10^{-1} $ & $ 1.747\times 10^{-7}  $ &  --       & $ 5.192\times 10^{-6} $ &  --       & $ 2.446\times 10^{-5} $ &  --       & $ 7.838\times 10^{-4} $ &  --       \\
 $ 8.333\times 10^{-2} $ & $ 2.303\times 10^{-8}  $ & $ 4.998 $ & $ 6.988\times 10^{-7} $ & $ 4.946 $ & $ 4.829\times 10^{-6} $ & $ 4.001 $ & $ 1.550\times 10^{-4} $ & $ 3.997 $ \\
 $ 6.250\times 10^{-2} $ & $ 5.466\times 10^{-9}  $ & $ 4.999 $ & $ 1.672\times 10^{-7} $ & $ 4.972 $ & $ 1.527\times 10^{-6} $ & $ 4.001 $ & $ 4.904\times 10^{-5} $ & $ 4.000 $ \\
 $ 5.000\times 10^{-2} $ & $ 1.791\times 10^{-9}  $ & $ 4.999 $ & $ 5.499\times 10^{-8} $ & $ 4.983 $ & $ 6.254\times 10^{-7} $ & $ 4.001 $ & $ 2.009\times 10^{-5} $ & $ 4.000 $ \\
 $ 4.167\times 10^{-2} $ & $ 7.200\times 10^{-10} $ & $ 5.000 $ & $ 2.214\times 10^{-8} $ & $ 4.989 $ & $ 3.016\times 10^{-7} $ & $ 4.001 $ & $ 9.687\times 10^{-6} $ & $ 4.000 $ \\
 $ 3.571\times 10^{-2} $ & $ 3.331\times 10^{-10} $ & $ 5.000 $ & $ 1.026\times 10^{-8} $ & $ 4.992 $ & $ 1.628\times 10^{-7} $ & $ 4.001 $ & $ 5.229\times 10^{-6} $ & $ 4.000 $ \\
 $ 3.125\times 10^{-2} $ & $ 1.709\times 10^{-10} $ & $ 5.000 $ & $ 5.266\times 10^{-9} $ & $ 4.994 $ & $ 9.541\times 10^{-8} $ & $ 4.000 $ & $ 3.065\times 10^{-6} $ & $ 4.000 $ \\
 \hline
\end{tabular}
}
\caption{Experimental convergence rate $r$ in $L^2$ and $H^1$ norms with $\mathcal{P}^4_h$ elements}\label{tab:eoc4}
\end{table}

Tables \ref{tab:eoc1}--\ref{tab:eoc4} present the recorded errors and the corresponding convergence rates for various uniform grids. We observe that the experimental convergence rates are optimal in all cases indicating that the numerical solution in the finite element space $\mathcal{P}^r$ satisfies the estimate
$$\|\eta-\teta\|+\|\phi-\tphi\|+h [\|\eta-\teta\|_1+\|\phi-\tphi\|_1]=O(h^{r+1})\ .$$
This estimate provides also the convergence of the velocity field, in the sense that
$$\|\bu-\tilde{\bu}\|=\|\nabla\phi-\nabla\tphi\|\leq \|\phi-\tphi\|_1=O(h^r)\ .$$

It is worth mentioning that we repeated the experiments with exact solution $\eta(\bx,t)=e^t\cos(\pi x)\cos(\pi y)$, keeping the rest of the data the same, which resulted in similar results and the same conclusions.

\subsection{Stability of the numerical method for large values of $\Delta t$}

To ensure that the semidiscrete system (\ref{eq:mixedform}) is not stiff and that the relaxation Runge-Kutta method does not require any restrictive bounds on the timestep parameter $\Delta t$, we performed a series of experiments with large $\Delta t$. First we repeated the computation of the convergence rate of the previous section but with $\Delta t=h$ for the classical fourth order Runge-Kutta method. The results are presented in Tables \ref{tab:eoc5}--\ref{tab:eoc8}.

\begin{table}[t]
{\small
\begin{tabular}{ccccccccc}
\hline
 $h$  & $E_0[h,\phi]$ & $r$ & $E_0[h,\eta]$ & $r$ & $E_1[h,\phi]$ & $r$ & $E_1[h,\eta]$ & $r$ \\
 \hline
 $ 2.0\times 10^{-1} $ & $ 1.289\times 10^{-1}  $ &  --       & $ 1.421\times 10^{-1} $ &  --       & $ 1.010\times 10^{0} $ &  --       & $ 4.049\times 10^{0} $ &  --       \\
 $ 1.0\times 10^{-1} $ & $ 3.076\times 10^{-2}  $ & $ 2.067 $ & $ 3.407\times 10^{-2} $ & $ 2.061 $ & $ 5.011\times 10^{-1} $ & $ 1.011 $ & $ 2.009\times 10^{0} $ & $ 1.011 $ \\
 $ 5.0\times 10^{-2} $ & $ 7.613\times 10^{-3}  $ & $ 2.014 $ & $ 8.431\times 10^{-3} $ & $ 2.014 $ & $ 2.501\times 10^{-1} $ & $ 1.002 $ & $ 1.003\times 10^{0} $ & $ 1.003 $ \\
 $ 2.5\times 10^{-2} $ & $ 1.899\times 10^{-3}  $ & $ 2.003 $ & $ 2.103\times 10^{-3} $ & $ 2.004 $ & $ 1.250\times 10^{-1} $ & $ 1.001 $ & $ 5.012\times 10^{-1} $ & $ 1.001 $ \\
 $ 2.0\times 10^{-2} $ & $ 1.215\times 10^{-3}  $ & $ 2.001 $ & $ 1.345\times 10^{-3} $ & $ 2.001 $ & $ 1.000\times 10^{-1} $ & $ 1.000 $ & $ 4.010\times 10^{-1} $ & $ 1.000 $ \\
 \hline
\end{tabular}
\caption{Experimental convergence rate $r$ in $L^2$ and $H^1$ norms with $\mathcal{P}^1_h$ elements and $\Delta t=h$}\label{tab:eoc5}
}\end{table}

\begin{table}[t]
{\small
\begin{tabular}{ccccccccc}
\hline
 $h$  & $E_0[h,\phi]$ & $r$ & $E_0[h,\eta]$ & $r$ & $E_1[h,\phi]$ & $r$ & $E_1[h,\eta]$ & $r$ \\
 \hline
 $ 2.0\times 10^{-1} $ & $ 2.042\times 10^{-3}  $ &  --       & $ 1.244\times 10^{-2} $ &  --       & $ 8.128\times 10^{-2} $ &  --       & $ 6.372\times 10^{-1} $ &  --       \\
 $ 1.0\times 10^{-1} $ & $ 2.378\times 10^{-4}  $ & $ 3.103 $ & $ 1.749\times 10^{-3} $ & $ 2.831 $ & $ 2.045\times 10^{-2} $ & $ 1.991 $ & $ 1.633\times 10^{-1} $ & $ 1.965 $ \\
 $ 5.0\times 10^{-2} $ & $ 2.906\times 10^{-5}  $ & $ 3.032 $ & $ 2.268\times 10^{-4} $ & $ 2.946 $ & $ 5.119\times 10^{-3} $ & $ 1.998 $ & $ 4.097\times 10^{-2} $ & $ 1.995 $ \\
 $ 2.5\times 10^{-2} $ & $ 3.612\times 10^{-6}  $ & $ 3.009 $ & $ 2.864\times 10^{-5} $ & $ 2.986 $ & $ 1.280\times 10^{-3} $ & $ 2.000 $ & $ 1.025\times 10^{-2} $ & $ 1.999 $ \\
 $ 2.0\times 10^{-2} $ & $ 1.848\times 10^{-6}  $ & $ 3.003 $ & $ 1.468\times 10^{-5} $ & $ 2.995 $ & $ 8.192\times 10^{-4} $ & $ 2.000 $ & $ 6.559\times 10^{-3} $ & $ 2.000 $ \\
 \hline
\end{tabular}
\caption{Experimental convergence rate $r$ in $L^2$ and $H^1$ norms with $\mathcal{P}^2_h$ elements and $\Delta t=h$}\label{tab:eoc6}
}\end{table}

\begin{table}[t]
{\small
\begin{tabular}{ccccccccc}
\hline
 $h$  & $E_0[h,\phi]$ & $r$ & $E_0[h,\eta]$ & $r$ & $E_1[h,\phi]$ & $r$ & $E_1[h,\eta]$ & $r$ \\
 \hline
 $ 2.0\times 10^{-1} $ & $ 7.255\times 10^{-5}  $ &  --       & $ 7.875\times 10^{-4} $ &  --       & $ 4.013\times 10^{-3} $ &  --       & $ 6.552\times 10^{-2} $ &  --       \\
 $ 1.0\times 10^{-1} $ & $ 4.118\times 10^{-6}  $ & $ 4.139 $ & $ 4.705\times 10^{-5} $ & $ 4.065 $ & $ 4.999\times 10^{-4} $ & $ 3.005 $ & $ 8.201\times 10^{-3} $ & $ 2.998 $ \\
 $ 5.0\times 10^{-2} $ & $ 2.672\times 10^{-7}  $ & $ 3.946 $ & $ 2.905\times 10^{-6} $ & $ 4.018 $ & $ 6.245\times 10^{-5} $ & $ 3.001 $ & $ 1.025\times 10^{-3} $ & $ 3.000 $ \\
 $ 2.5\times 10^{-2} $ & $ 1.739\times 10^{-8}  $ & $ 3.942 $ & $ 1.810\times 10^{-7} $ & $ 4.004 $ & $ 7.804\times 10^{-6} $ & $ 3.000 $ & $ 1.281\times 10^{-4} $ & $ 3.000 $ \\
 $ 2.0\times 10^{-2} $ & $ 7.191\times 10^{-9}  $ & $ 3.957 $ & $ 7.411\times 10^{-8} $ & $ 4.002 $ & $ 3.996\times 10^{-6} $ & $ 3.000 $ & $ 6.558\times 10^{-5} $ & $ 3.000 $ \\
 \hline
\end{tabular}
\caption{Experimental convergence rate $r$ in $L^2$ and $H^1$ norms with $\mathcal{P}^3_h$ elements and $\Delta t=h$}\label{tab:eoc7}
}\end{table}

\begin{table}[t]
{\small
\begin{tabular}{ccccccccc}
\hline
 $h$  & $E_0[h,\phi]$ & $r$ & $E_0[h,\eta]$ & $r$ & $E_1[h,\phi]$ & $r$ & $E_1[h,\eta]$ & $r$ \\
 \hline
 $ 2.0\times 10^{-1} $ & $ 3.766\times 10^{-5}  $ &  --       & $ 6.016\times 10^{-5} $ &  --       & $ 2.871\times 10^{-4} $ &  --       & $ 5.112\times 10^{-3} $ &  --       \\
 $ 1.0\times 10^{-1} $ & $ 2.025\times 10^{-6}  $ & $ 4.217 $ & $ 2.678\times 10^{-6} $ & $ 4.489 $ & $ 1.811\times 10^{-5} $ & $ 3.987 $ & $ 3.221\times 10^{-4} $ & $ 3.988 $ \\
 $ 5.0\times 10^{-2} $ & $ 1.470\times 10^{-7}  $ & $ 3.784 $ & $ 1.429\times 10^{-7} $ & $ 4.228 $ & $ 1.140\times 10^{-6} $ & $ 3.989 $ & $ 2.014\times 10^{-5} $ & $ 3.999 $ \\
 $ 2.5\times 10^{-2} $ & $ 1.037\times 10^{-8}  $ & $ 3.826 $ & $ 8.541\times 10^{-9} $ & $ 4.064 $ & $ 7.161\times 10^{-8} $ & $ 3.993 $ & $ 1.259\times 10^{-6} $ & $ 4.000 $ \\
 $ 2.0\times 10^{-2} $ & $ 4.357\times 10^{-9}  $ & $ 3.886 $ & $ 3.483\times 10^{-9} $ & $ 4.020 $ & $ 2.936\times 10^{-8} $ & $ 3.995 $ & $ 5.156\times 10^{-7} $ & $ 4.000 $ \\
 \hline
\end{tabular}
\caption{Experimental convergence rate $r$ in $L^2$ and $H^1$ norms with $\mathcal{P}^4_h$ elements and $\Delta t=h$}\label{tab:eoc8}
}\end{table}

We observe that the convergence rates are the expected optimal rates. In the case of quartic elements $\mathcal{P}^4_h$, the convergence in space is fifth order as shown in the previous section while the convergence of the Runge-Kutta method is of fourth order. Therefore, the temporal error dominates, and as a result, we observe in Table \ref{tab:eoc8} that the convergence rates in the $L^2$ norm are of fourth order. This also serves as an indication that the convergence in time is of fourth order as expected.

Furthermore, we studied the propagation and reflection of the analytical solitary wave solution (\ref{eq:solitwave1})--(\ref{eq:solitwave2}) of the Bona-Smith system with $\theta^2 = 9/11$, $D_0 = 1$, and $g = 1$ in a rectangular channel $[-50, 50] \times [-5, 5]$ of depth $D_0 = 1$. The analytical solitary wave of this particular system has an amplitude $A = 1$, and thus everything scaled to $O(1)$ can provide adequate evidence of the presence of possible stiffness.

For the spatial discretization, we considered a regular (unstructured) triangulation of the domain of $55080$ triangles with the maximum and minimum triangle sides being $\max(h_i) \approx 0.30$ and $\min(h_i) \approx 0.15$, respectively. For the temporal discretization, we used the Relaxation Runge-Kutta method of order four with $\Delta t = 0.1, 0.4, 0.5$, and $1$ where we integrated the system up to $T=100$. For $\Delta t = 2$, the numerical method became unstable and the Relaxation Runge-Kutta method failed to converge for all finite element spaces; however, for all the other values of $\Delta t$, the solution remained stable for all times, indicating that the semidiscretization is not stiff and no serious restrictions on $\Delta t$ are required except perhaps that $\Delta t < C$ where $C$ a constant greater than 1. 

The mass and energy for $\mathcal{P}_h^1$ elements were calculated as $\mathcal{M} = 31.13995776646$ and $\mathcal{E} = 24.890739171$, respectively, and remained constant to the displayed decimal places for all tested $\Delta t$ values. Similarly, for $\mathcal{P}_h^2$ elements, we obtained $\mathcal{M} = 31.13995776646$ and $\mathcal{E} = 24.911946804$. Similar conservation properties were observed for higher-order elements. While the numerical solutions exhibited stability for  $\Delta t=1$, the time integration errors became significant. Consequently, we restrict our subsequent numerical experiments to relatively small temporal mesh length, $\Delta t < 1$. 

\subsection{Propagation and reflection of shoaling solitary waves}\label{sec:expsw}

We tested the new method and the Bona-Smith system ($\theta^2 = 1$, $g=9.81~m/s^2$) against the propagation and reflection of line solitary waves using two laboratory experiments of \cite{Dodd1998}. We numerically generated solitary waves with amplitudes  $A_1 = 0.07 \,m$ and $A_2 = 0.12 \,m$ using the Petviashvili method in the domain  $\Omega = [-50, 20] \times [0, 1]$ with bottom topography described by the function
$$
D(x, y) = \left\{
\begin{array}{ll}
0.7, & x \leq 0 \\
0.7 - x/50, & x > 0
\end{array}
\right., \quad 0 \leq y \leq 1.
$$
Note that in this experiment we used the original bathymetry without smoothing its corner. Figure \ref{fig:reflxstp} presents the initial solitary wave with amplitude $A_1$ and the bottom topography $-D(x,y)$. 

\begin{figure}[t]
\centering
\includegraphics[width=0.6\columnwidth]{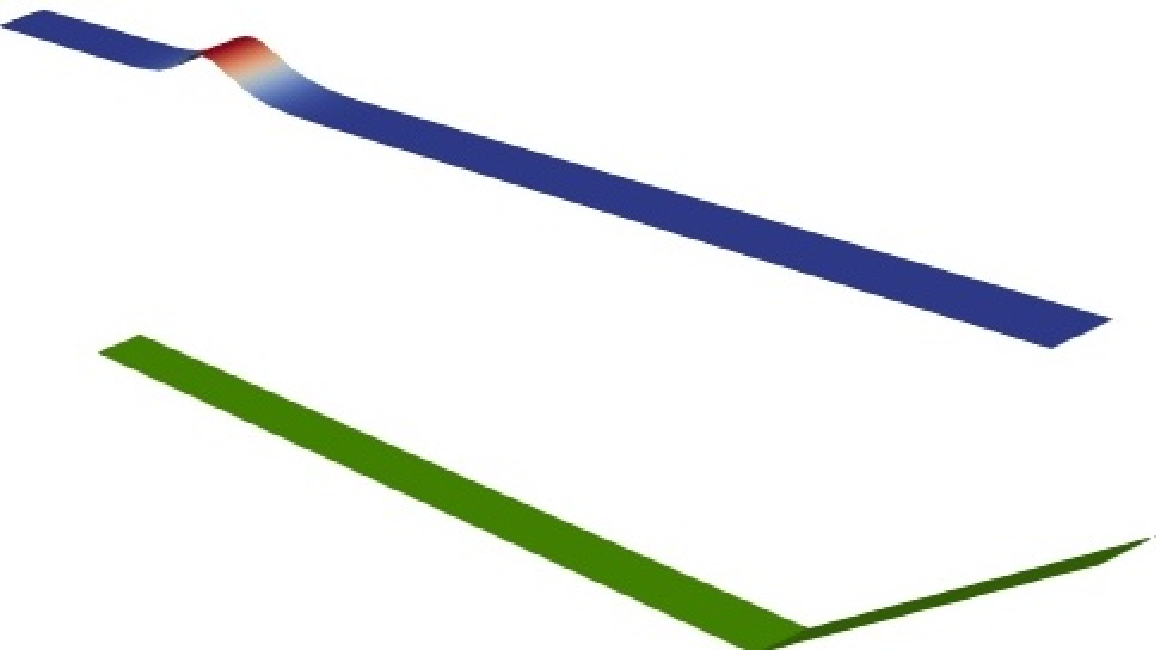}
\caption{Reflection of shoaling solitary waves: Initial setup}
\label{fig:reflxstp}
\end{figure}

All variables used in this experiment are dimensional and expressed in SI units. The traveling waves were horizontally translated such that their maximum at $t = 0$ is achieved at $x_0 = -30 \,m$. These waves traveled to the right until they reached the sloping bottom, where they began shoaling. Then they were reflected by the solid wall, located at $x = 20 \,m$, and propagated to the right followed by dispersive tails. The reflection of water waves on a vertical, solid wall is equivalent to the head-on collision of symmetric counter-propagating waves, and thus the reflection is inelastic \cite{ADM2010}. We recorded the free-surface elevation $\eta$ at three gauges located at $(x,y)=(g_i,0)$ for $i=,1,2,3$ where $g_1 = 0$, $g_2 = 16.25$, and $g_3 = 17.75$.

\begin{figure}[t]
\centering
\includegraphics[width=\columnwidth]{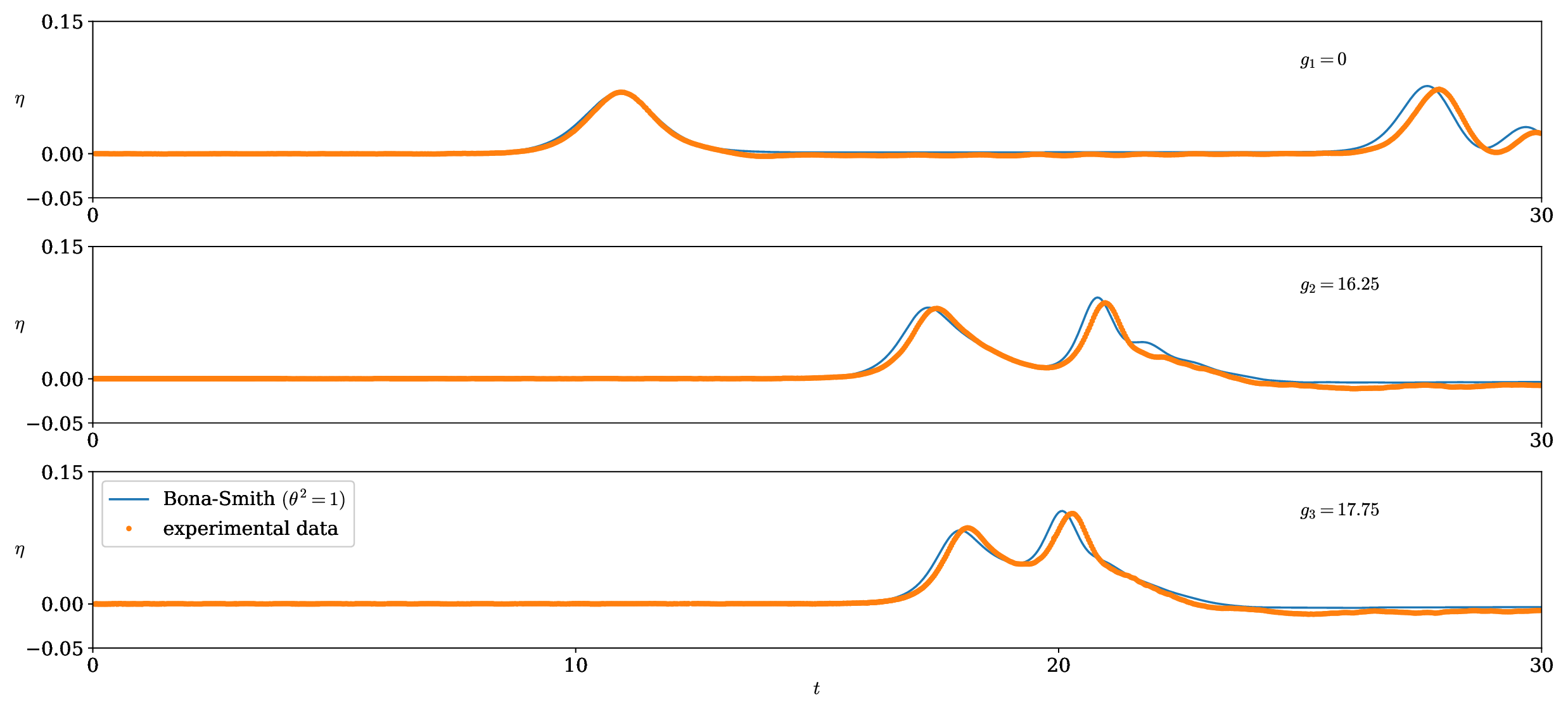}
\caption{Reflection of shoaling solitary wave of amplitude $A_1=0.07~m$ by a solid wall }
\label{fig:refl1}
\end{figure}

\begin{figure}[t]
\centering
\includegraphics[width=\columnwidth]{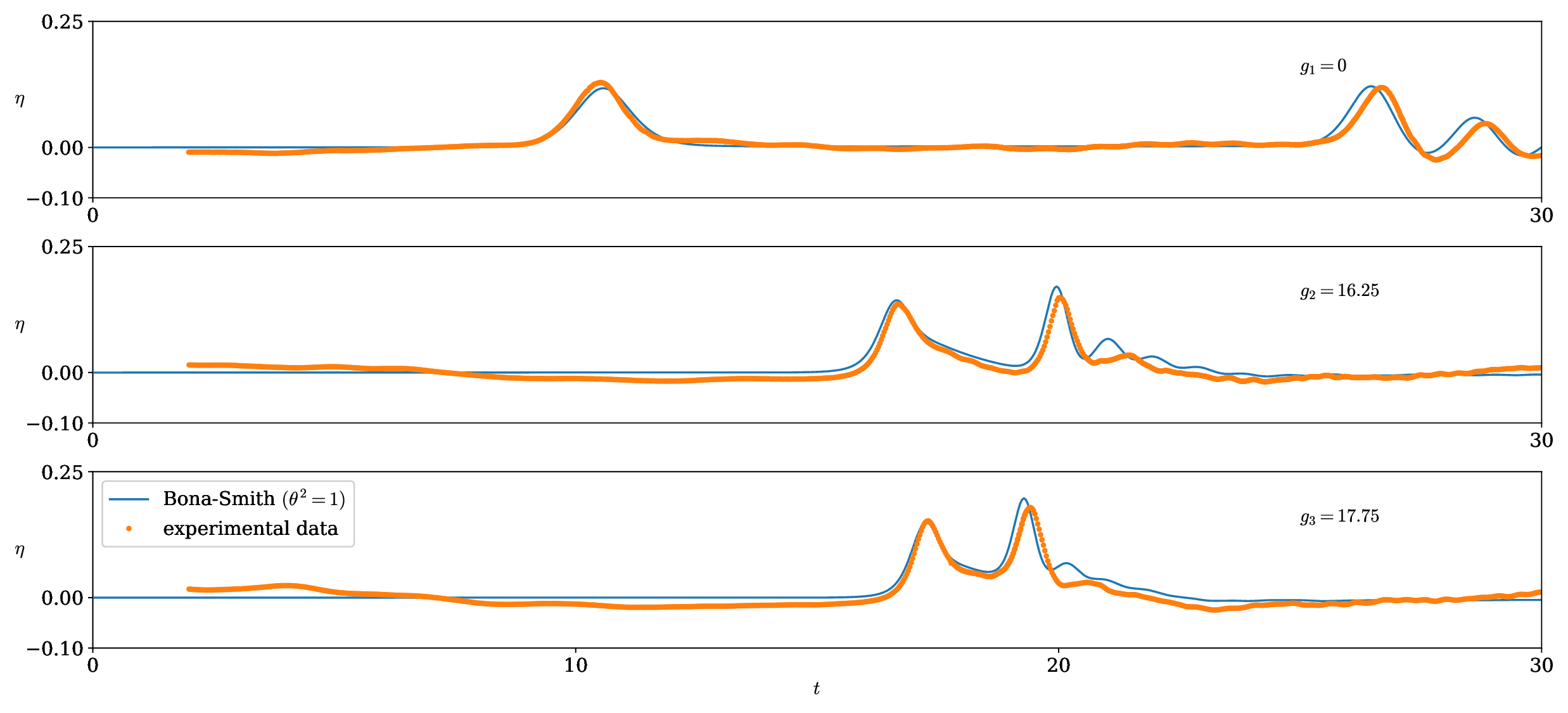}
\caption{Reflection of shoaling solitary wave of amplitude $A_2=0.12~m$ by a solid wall }
\label{fig:refl2}
\end{figure}

In Figures \ref{fig:refl1}--\ref{fig:refl2}, we present a comparison between the laboratory data from \cite{Dodd1998} and the numerical results obtained using a new numerical method for solving the Bona-Smith system with $\theta^2 = 1$. For time discretization, we used $\Delta t = 0.04$. For spatial discretization, we utilized $\mathcal{P}^1_h$ elements for all variables and employed a regular triangulation of the domain $\Omega$ with $N = 3,676$ elements, each with a maximum edge length of approximately $h \approx 0.33$. The agreement observed between the numerical results and the laboratory data corroborates previous findings related to the BBM-BBM and Nwogu systems, suggesting that the dispersion properties of the model are not critical in this experiment \cite{ADM2010}. 

During these experiments, mass and energy were conserved up to the precision shown: when $A_1 = 0.07$, the mass $\mathcal{M}$ was $0.37465842341571$ and the energy $\mathcal{E}$ was $0.17465474989439$; when $A_2 = 0.12$, the mass $\mathcal{M}$ was $0.5049385982123$ and the energy $\mathcal{E}$ was $0.40786323559272$.

In Figure \ref{fig:gamma1}, we depict the difference $\gamma^n-1$ of the relaxation parameter as a function of $t$. In both cases, the parameter $\gamma^n$ remained within the predicted range of $O(\Delta t^3)$ and did not exceed the value of $1+10^{-4}$. However, its value increases notably when the solitary wave interacts with the bottom topography.

\begin{figure}[t]
\centering
\includegraphics[width=\columnwidth]{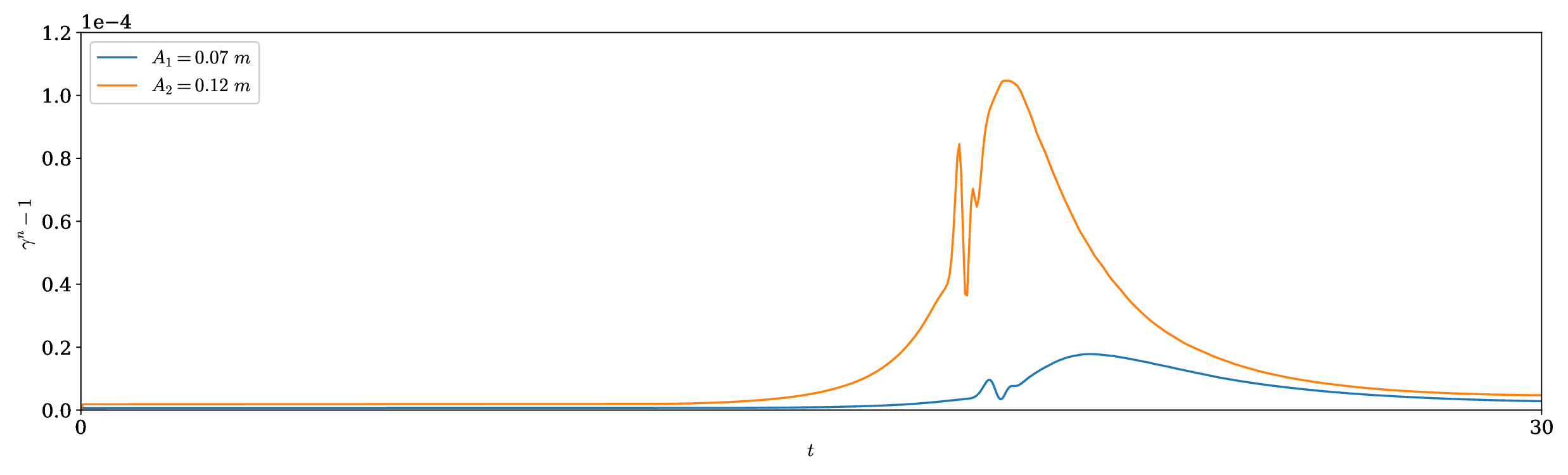}
\caption{Reflection of shoaling solitary waves: Difference $\gamma^n-1$ of the relaxation parameter as a function of $t$}
\label{fig:gamma1}
\end{figure}

\subsection{Solitary wave interaction with vertical cylinder}\label{sec:cylinder}

In the next experiment, we study the interaction of a solitary wave of the Bona-Smith system ($\theta^2=1$, $g=9.81~m/s^2$) with a vertical cylinder \cite{ASB1993}. For this experiment, we consider a numerical solitary wave with speed $c_s= 1.356$ and amplitude of $A\approx 0.036$  propagating over a flat bottom with $D_0=0.15$ in a domain $\Omega=[-4,20]\times [0,0.55]\setminus \mathcal{C}$, where $\mathcal{C}=\{(x,y)\in\mathbb{R}^2~:~(x-4.5)^2+(y-0.275)^2\leq 0.08^2\}$ represents a circle with center $(4.5,0.275)$ and radius $0.08~m$. This domain simulates a channel with a vertical cylinder.

\begin{figure}[t]
\centering
\includegraphics[width=0.4\columnwidth]{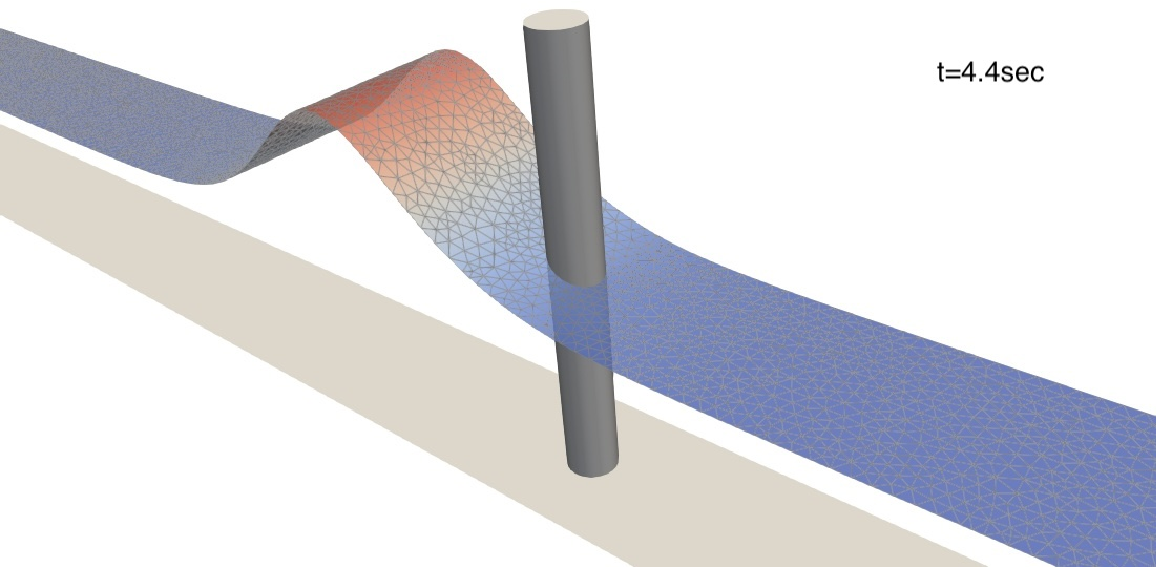}
\includegraphics[width=0.4\columnwidth]{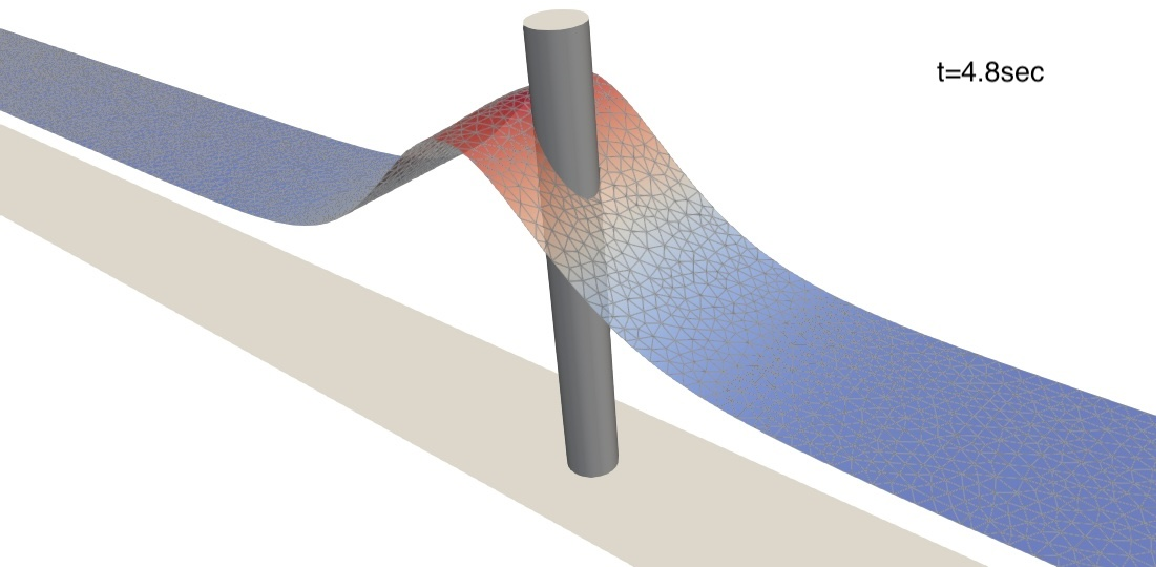}\\
\includegraphics[width=0.4\columnwidth]{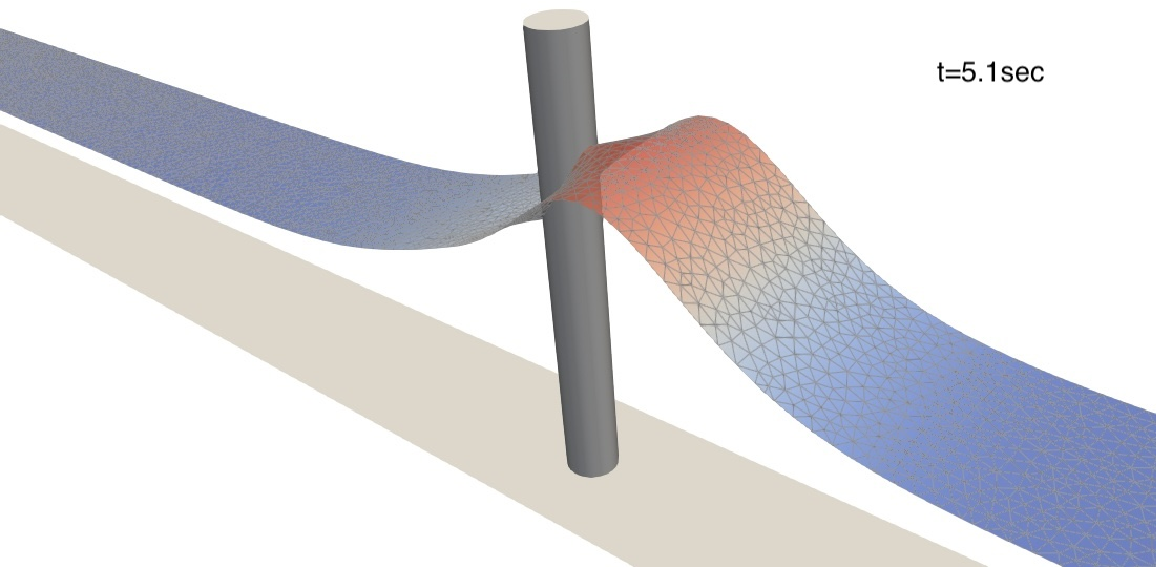}
\includegraphics[width=0.4\columnwidth]{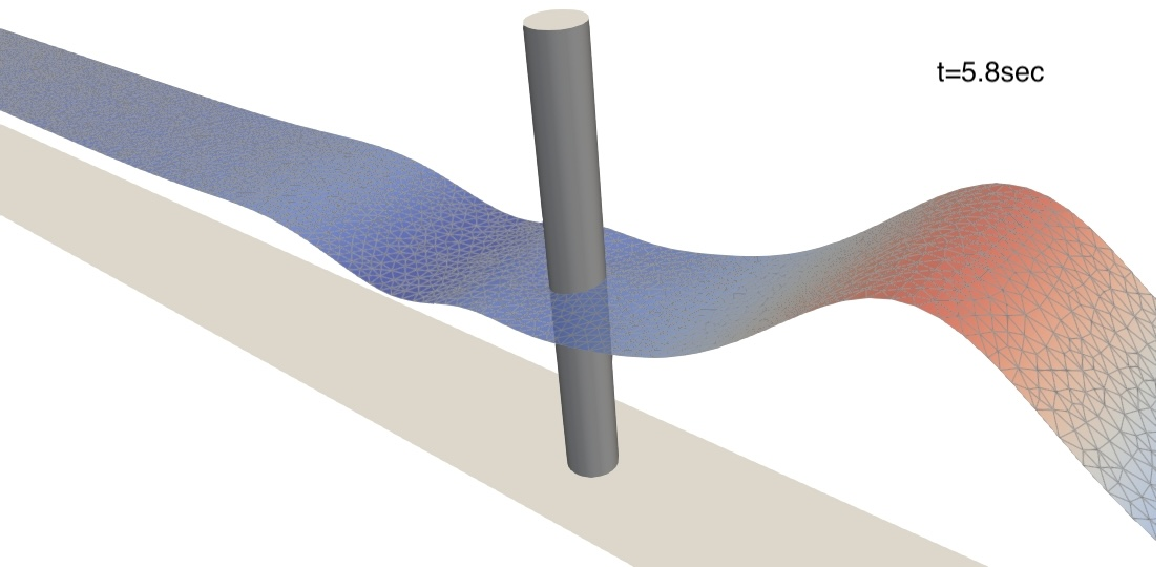}
\caption{Interaction of solitary wave with cylinder}
\label{fig:cyl1}
\end{figure}

\begin{figure}[t]
\centering
\includegraphics[width=\columnwidth]{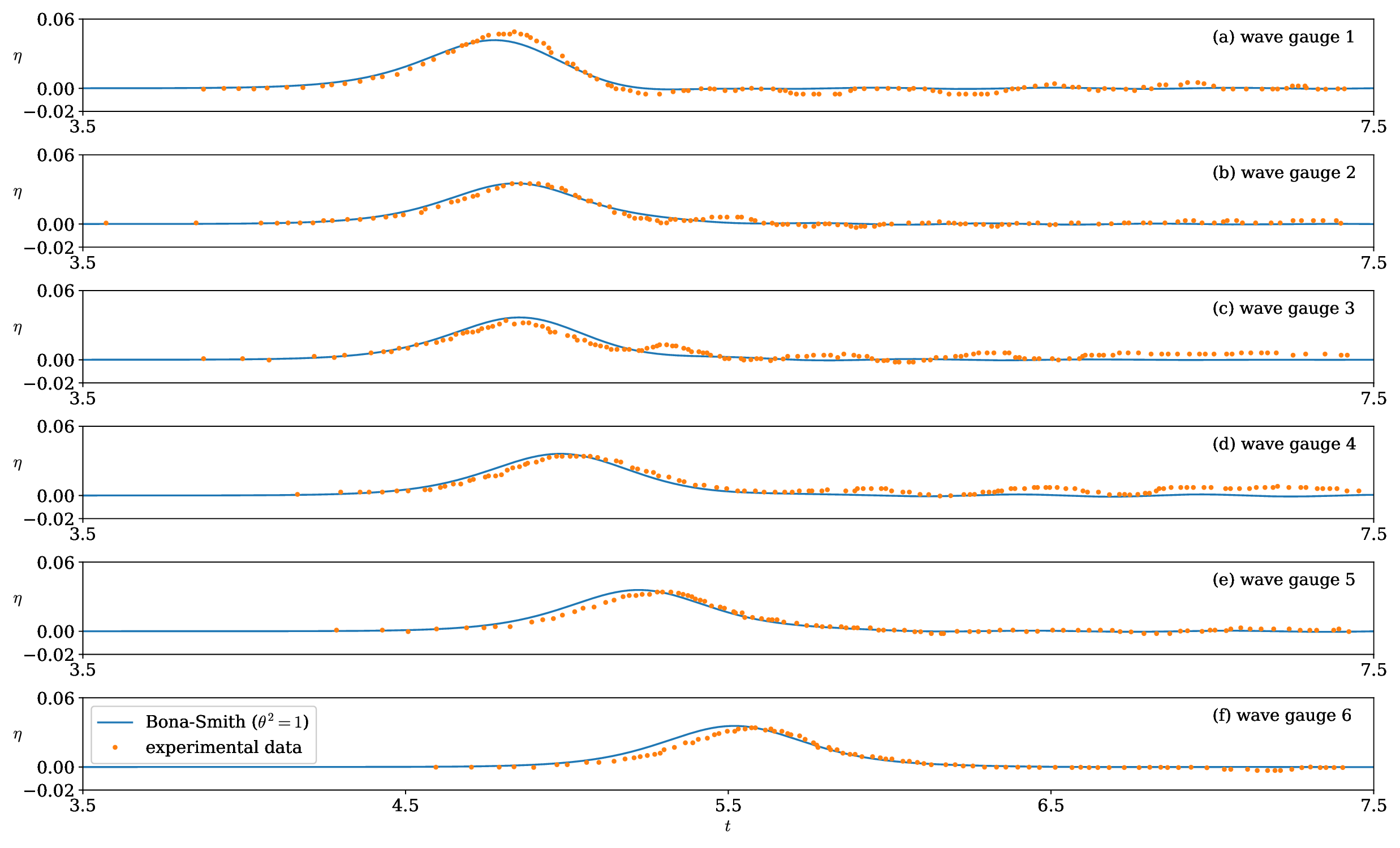}
\caption{Scattering of solitary wave by cylinder: Recorded solution at six wave gauges}
\label{fig:cyl2}
\end{figure}
\begin{figure}[t]
\centering
\includegraphics[width=\columnwidth]{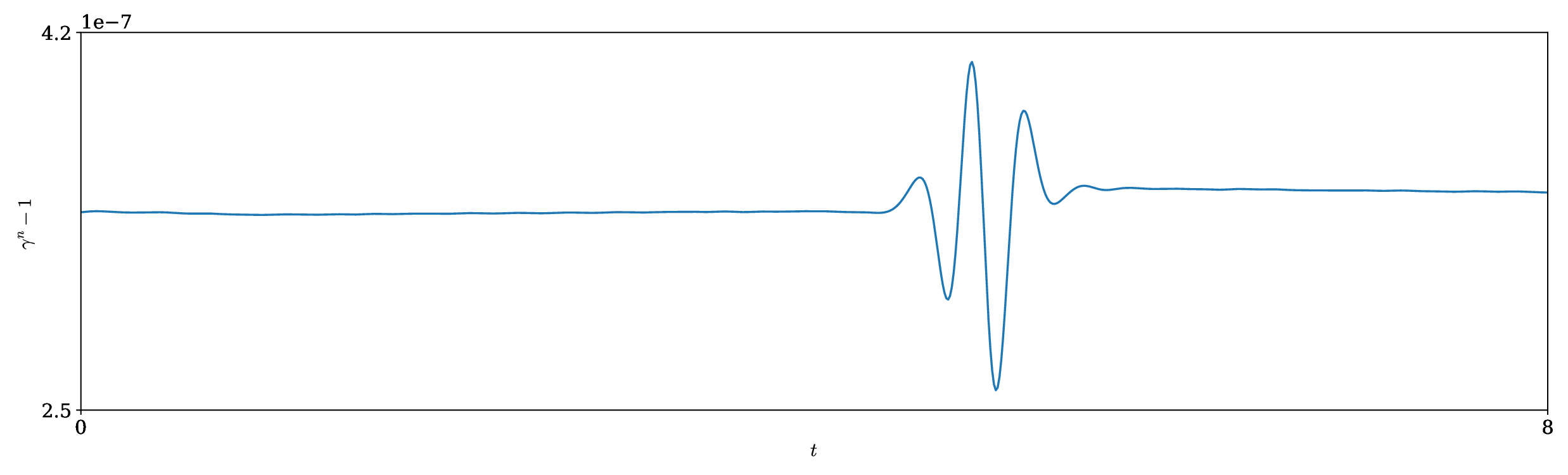}
\caption{Scattering of solitary wave by cylinder: Difference $\gamma^n-1$ of the relaxation parameter as a function of $t$}
\label{fig:gamma2}
\end{figure}

This experiment assumes frictionless walls, necessitating slip-wall conditions. The solution was recorded at six wave gauges located at positions: $g_1 = (4.4, 0.271)$, $g_2 = (4.5, 0.17)$, $g_3 = (4.5, 0.045)$, $g_4 = (4.6, 0.275)$, $g_5 = (4.975, 0.275)$, and $g_6 = (5.375, 0.275)$. For time discretization, we used $\Delta t = 0.01$. The domain $\Omega$ was discretized using a regular triangular grid consisting of $22,285$ triangles, each with a maximum edge length of $h = 0.059$. The circular obstacle was approximated by a regular icosagon. Spatial discretization utilized $\mathcal{P}^1_h$ elements for all unknowns. The initial conditions for $\eta$ and $\bu$ were generated using the Petviashvili method and adjusted so that the wave's maximum was positioned at $x = -2.1$.

The simulation remained stable and resulted in a smooth solution. In Figure \ref{fig:cyl1}, we present the interaction of the solitary wave with the cylinder. In Figure \ref{fig:cyl2}, we present the recorded solution at the six wave gauges $g_i$, $i=1,\dots,6$. We observe that contrary to the results obtained by the BBM-BBM system, the Bona-Smith system does not result in large oscillations around the cylinder. This characteristic of the Bona-Smith system has also been observed in the case of sticky-wall boundary conditions \cite{DMS2010}.

In this experiment, the mass was $\mathcal{M}=0.015807360969348$, the energy was $\mathcal{E}=0.0040425386059991$ (preserving the shown digits), while the vorticity was $0.0$. The parameter $\gamma^n$ remained close to $1$, with $\gamma^n-1=O(10^{-7})$ as expected. The recorded values of the difference $\gamma^n-1$ are presented in Figure \ref{fig:gamma2} as a function of time. We observe that during the interaction of the solitary wave with the cylinder, the parameter $\gamma^n$ fluctuates and then stabilizes close to a slightly higher value compared to its value before the interaction.

\subsection{Periodic waves over a submerged bar}\label{sec:perdc}

In this experiment, we explore a challenging scenario for Boussinesq systems, as detailed in \cite{BB1994}. We examine the dynamics when a small-amplitude periodic wave train travels over a flat bottom and subsequently interacts with a submerged bar within a channel. As these periodic waves encounter the rising slope of the submerged bar, they shoal and steepen, resulting in the generation of higher harmonics. Subsequently, the waves accelerate down the descending slope, generating a broad spectrum of wave-numbers, thereby testing the limits of Boussinesq models. For the initial conditions we consider the functions
$$\begin{aligned}
&\eta(\bx,0)=A\cos(k(x-x_0))(1-\tanh(x))(1+\tanh(x+x_1))/4\ ,\\
&\bu(\bx,0)=\tfrac{g}{D_0}\left[\eta(\bx,0)-\tfrac{1}{4D_0}\eta^2(\bx,0)-cD_0^2\eta_{xx}(\bx,0)\right]\ ,
\end{aligned}$$
where $A=0.01~m$, $g=9.81~m/s^2$, $D_0=0.4~m$, $k=1.67$, $x_0=3.6~m$, and $x_1=60~m$.
The bottom topography is described by the depth function
$$
D(\bx)=\left\{\begin{array}{ll}
-0.05x+0.7, & x\in[6,12)\\
0.1, & x\in[12,14)\\
0.1 x -1.3, & x\in[14,17]\\
0.4, & \text{elsewhere}
\end{array}\right.\ .
$$
Note that in this experiment we used the original bottom function without smoothing its corners. The initial conditions for the free-surface elevation and the bathymetry are presented in Figure \ref{fig:bottomm}. The formula for the velocity $\bu$ can be derived using asymptotic techniques and the BBM trick assuming unidirectional wave propagation \cite{BC1998}, resulting in waves that mainly propagate in one direction.

\begin{figure}[t]
\centering
\includegraphics[width=0.6\columnwidth]{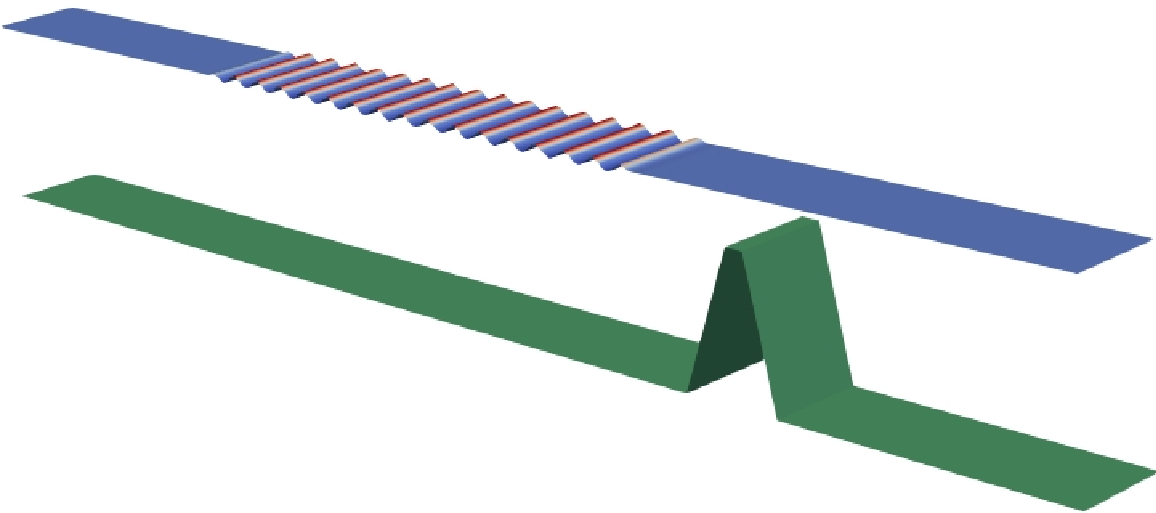}
\caption{Initial periodic wave-train and bottom topography}
\label{fig:bottomm}
\end{figure}
\begin{figure}[t]
\centering
\includegraphics[width=\columnwidth]{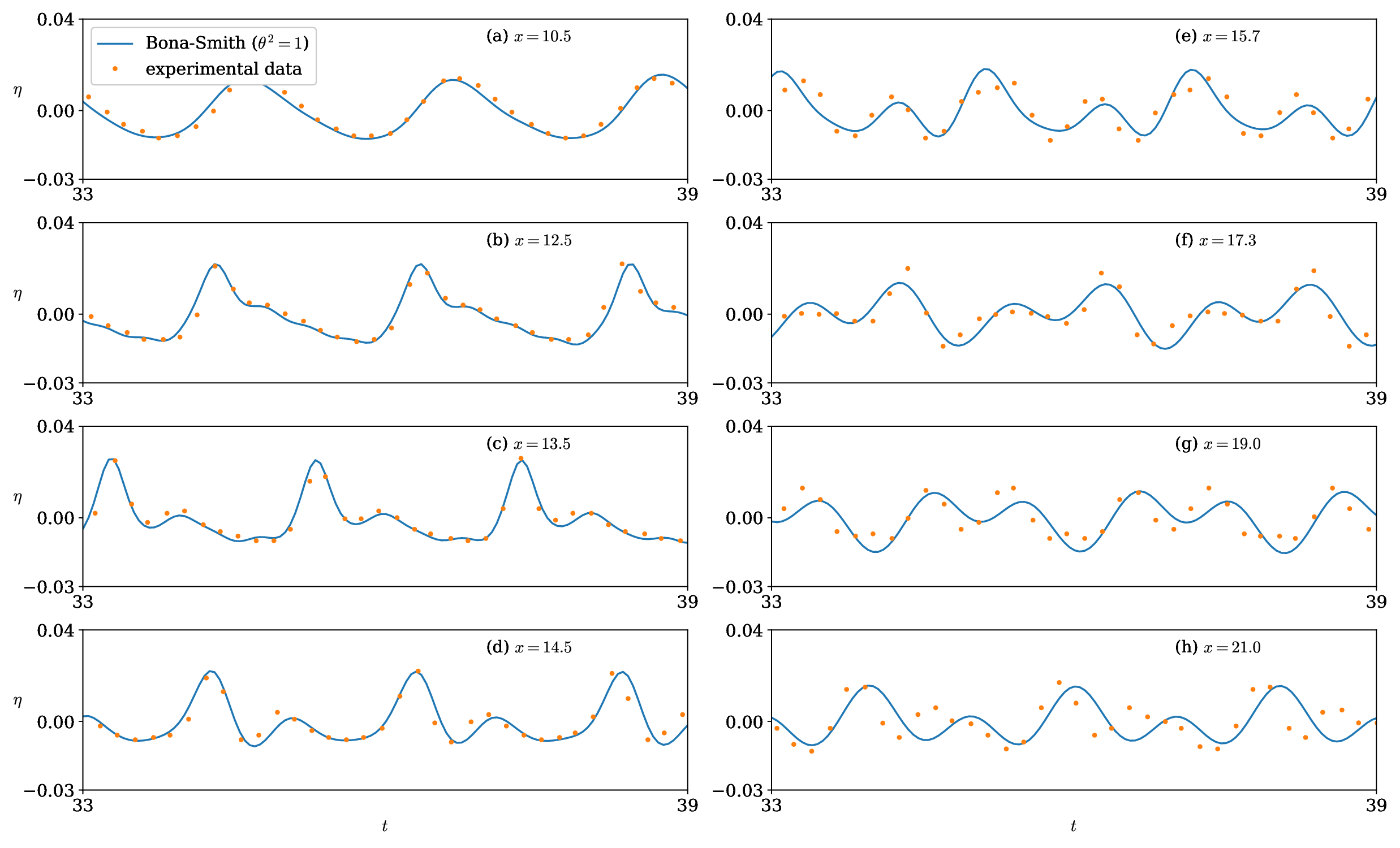}
\caption{Periodic waves over a bar: Free-surface elevation at various wave gauges}
\label{fig:peridig}
\end{figure}
\begin{figure}[t]
\centering
\includegraphics[width=\columnwidth]{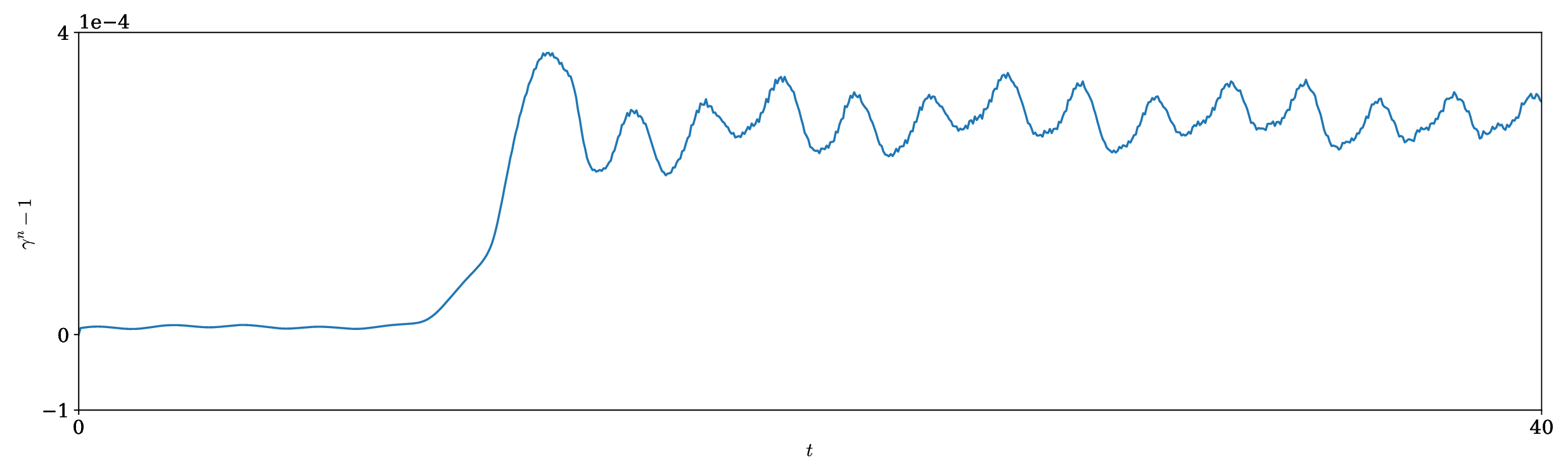}
\caption{Periodic waves over a bar: Difference $\gamma^n-1$ of the relaxation parameter as a function of $t$}
\label{fig:gammaper}
\end{figure}

In this experiment, we numerically integrated the Bona-Smith system ($\theta^2=1$, $g=9.81~m/s^2$) within the domain $\Omega=[-100,50]\times [0,1]$ until $T=40~s$. The domain was discretized using a uniform triangulation of 24,000 right-angle triangles, with the maximum side length of the triangles approximately $0.16$. For spatial discretization, we employed $\mathcal{P}^2_h$ elements for all the unknowns, and for time integration, we used a time step $\Delta t=0.05$. We recorded the free-surface elevation at eight wave gauges located at the points $(x,0.5)$ for $x=10.5, 12.5, 13.5, 14.5, 15.7$, and $17.3$. Figure \ref{fig:peridig} presents a comparison of the recorded solution with the laboratory data from \cite{BB1994}. The agreement between the numerical and laboratory data is satisfactory, with the numerical solution remaining close to the experimental data at all wave gauges, especially at the first four, where the numerical solution nearly interpolates the experimental data.  The numerical results of other Bona-Smith systems with $\theta^2<1$ are very similar, but the one with $\theta^2=1$ provides the best approximation for this set of data \cite{IKKM2021,KMS2020}. This improvement is not solely attributable to the new numerical method but also to the improved dispersion characteristics of the particular Bona-Smith system. Systems with enhanced dispersion characteristics, such as Nwogu's system \cite{Nwogu93}, the improved Serre-Green-Naghdi equations \cite{CDM2017,Lannes2013}, and the non-hydrostatic systems of \cite{MMS1991,BGSM2011}, perform better in this experiment because the solution consists of waves with high wave-numbers. 

The computed mass of the numerical solution is $\mathcal{M}=-0.0007507020714$, and the computed energy is $\mathcal{E}=0.035718392272188$, preserving the digits shown, while the vorticity was less than $5\times 10^{-19}$. The negative sign in the mass is because the mass was computed with respect to the variable $\eta$ and not the total depth $D+\eta>0$.

Once again, the parameter $\gamma^n$ remains at $1 + O(\Delta t^3)$. A graph depicting the difference $\gamma^n - 1$ is presented in Figure \ref{fig:gammaper}. As the waves begin interacting with the bottom topography, we observe a slight increase in the value of the parameter $\gamma^n$.

\subsection{Solitary wave scattering at a Y-junction}

Solitary wave scattering at a junction involves the interaction of solitary waves with the geometric features of the channel, specifically at points where the channel splits or merges. In general, as a solitary wave approaches a junction, the change in channel geometry can lead to a variety of behaviors including reflection, transmission, and splitting of the wave. When the wave reaches the junction, part of the wave may transmit into the new channels while part may reflect backward. The exact behavior can depend on the wave’s amplitude, speed, and the channel's properties, \cite{NDS2012}.

Here, we consider a channel with a horizontal bottom that splits into two symmetric channels. The Y-junction forms an angle $\theta\approx 60^\circ$. This setup presents considerable challenges, as solutions to other models such as the non-dispersive shallow water equations tend to form singularities (discontinuities) at the corners of the junction. However, in our case, due to the regularization properties of the Bona-Smith system, the numerical solution remains stable in such scenarios and no spurious oscillations result using continuous finite element methods. 

\begin{figure}[t]
\centering
\includegraphics[width=0.5\columnwidth]{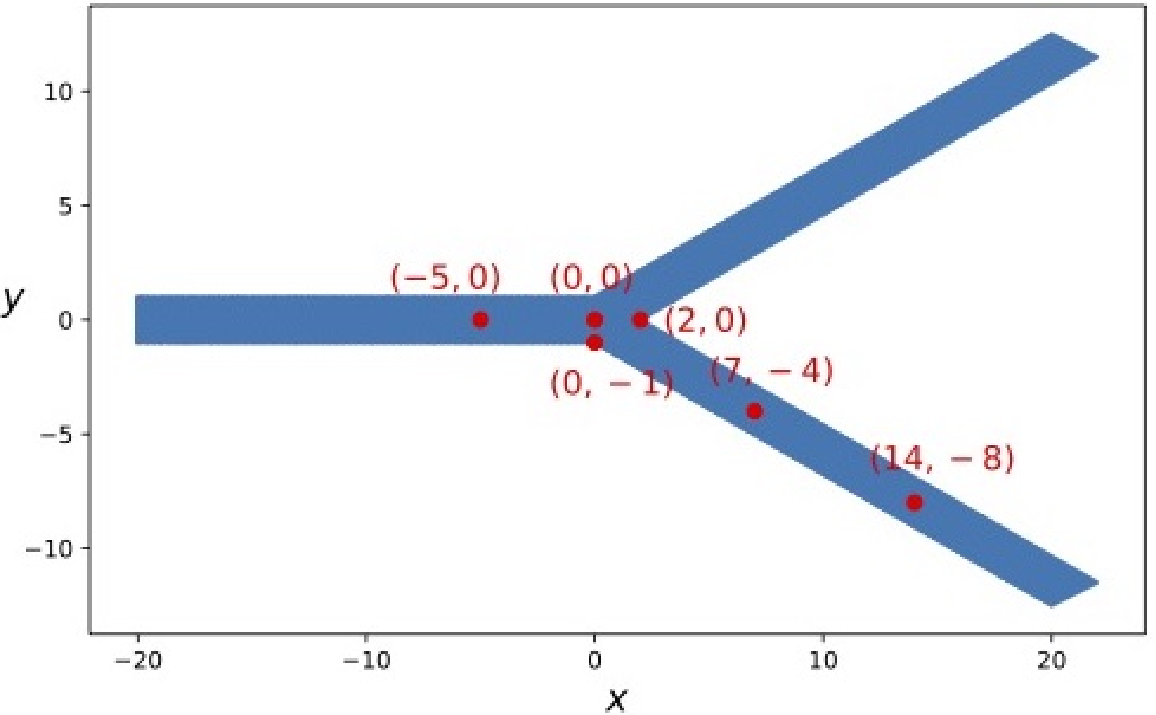}
\caption{Domain $\Omega$ depicting a channel with a junction, along with the locations of the wave gauges}
\label{fig:domainj}
\end{figure}
\begin{figure}[t]
\centering
\includegraphics[width=0.4\columnwidth]{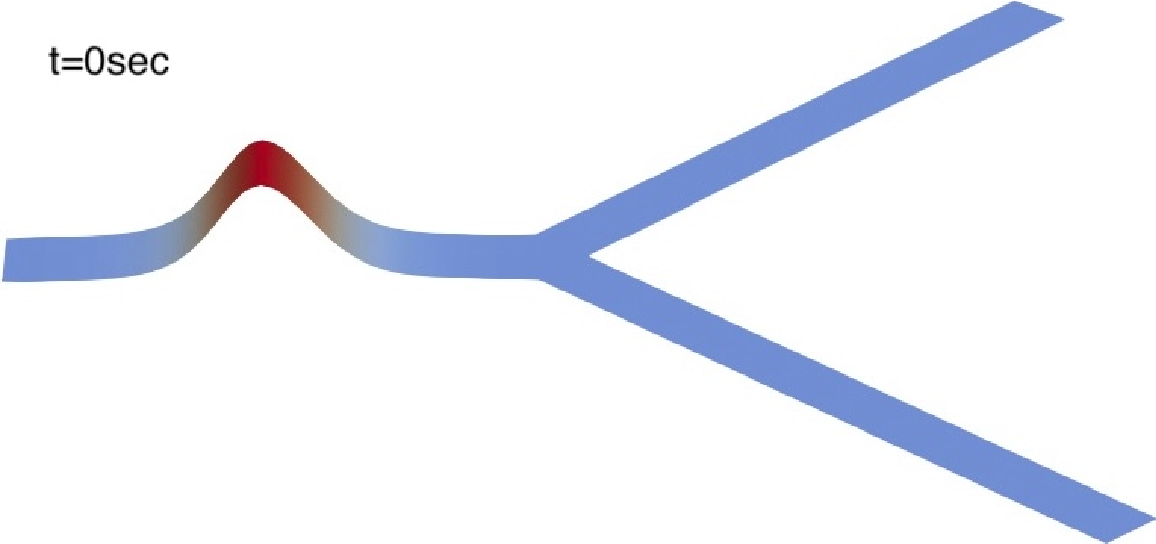}
\includegraphics[width=0.4\columnwidth]{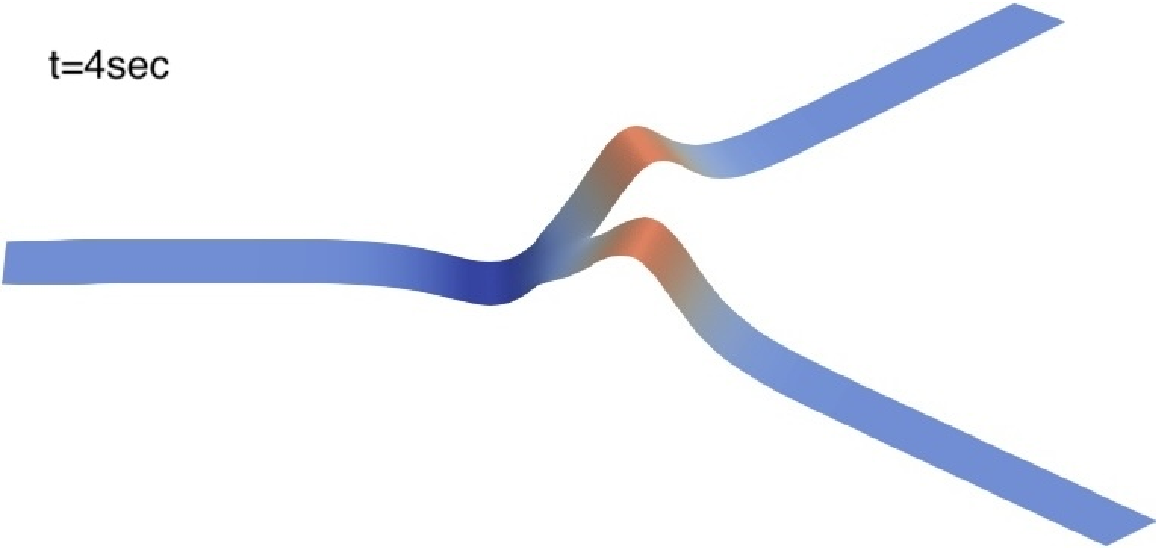}\\
\includegraphics[width=0.4\columnwidth]{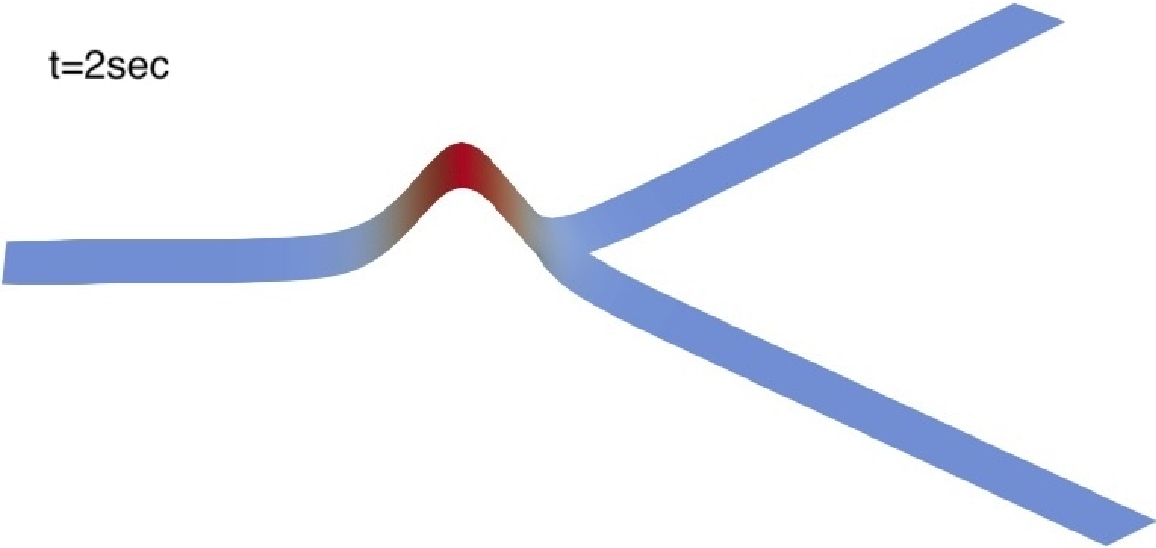}
\includegraphics[width=0.4\columnwidth]{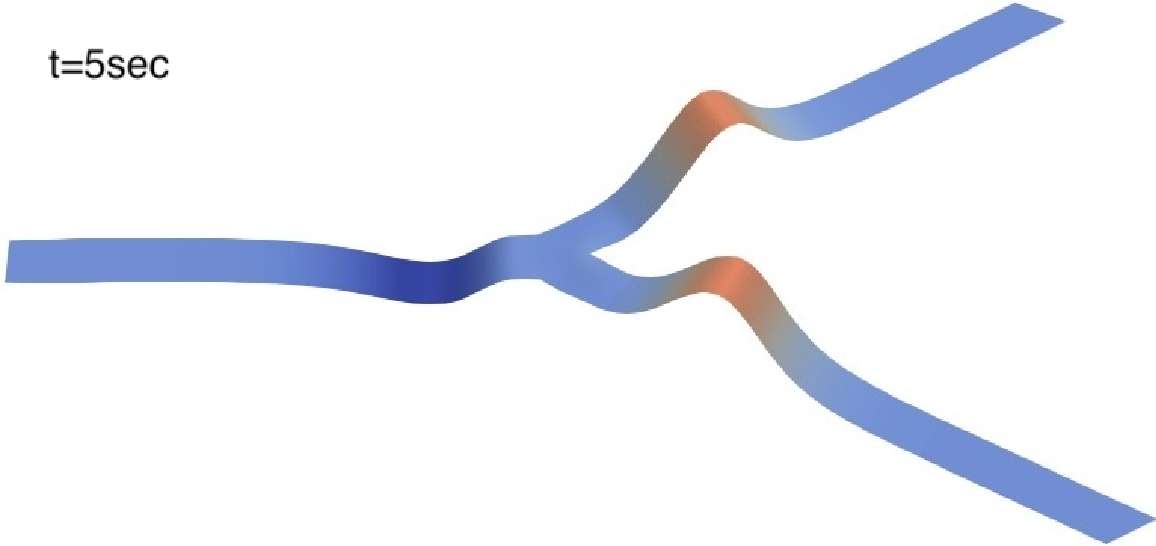}\\
\includegraphics[width=0.4\columnwidth]{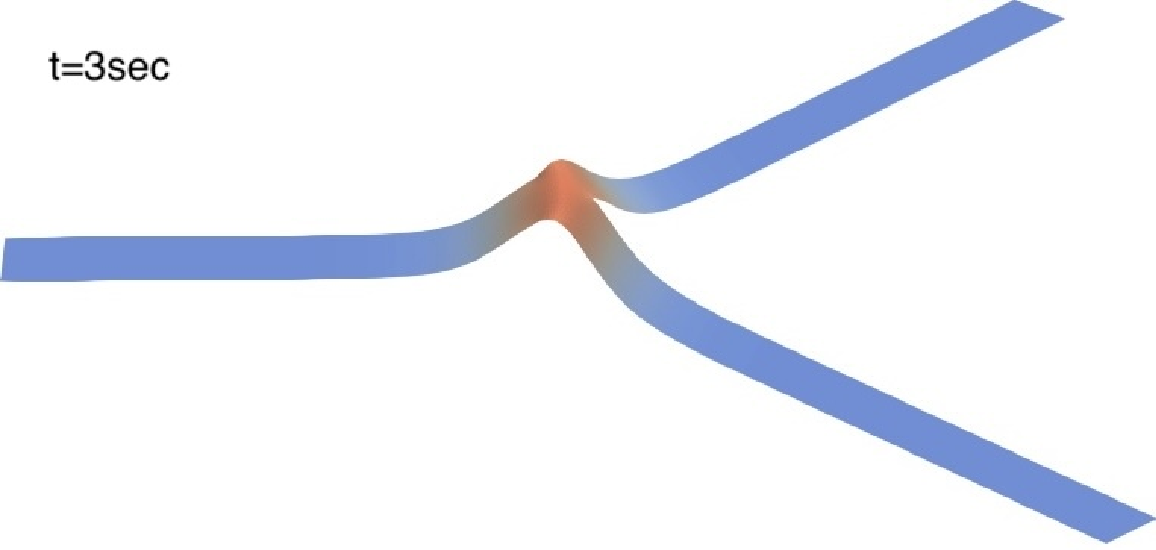}
\includegraphics[width=0.4\columnwidth]{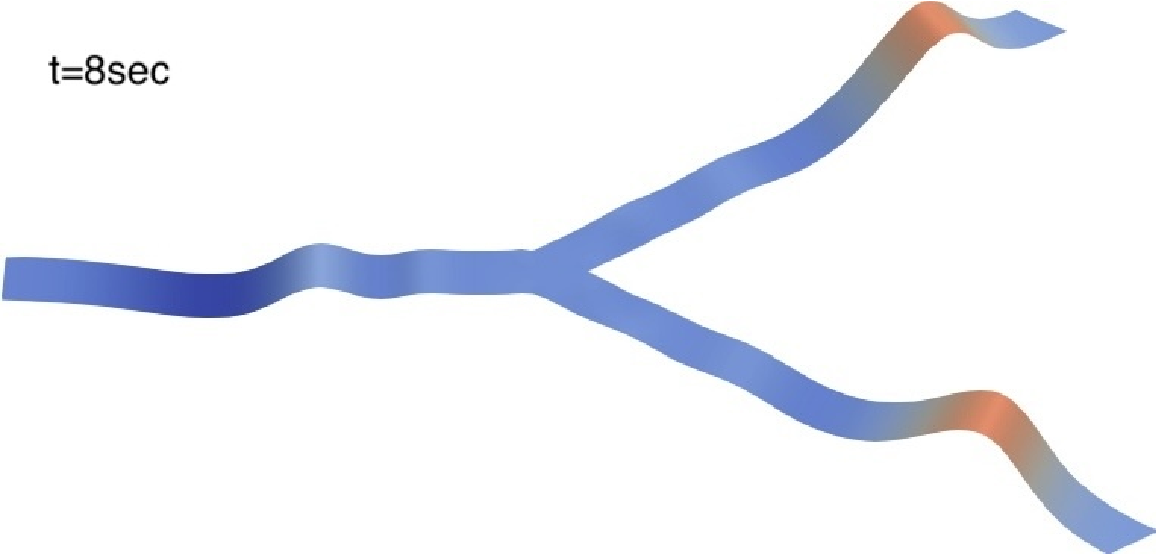}
\caption{Scattering of solitary wave at a Y-junction. Bona-Smith system with $\theta^2=1$}
\label{fig:junc}
\end{figure}
\begin{figure}[t]
\centering
\includegraphics[width=0.4\columnwidth]{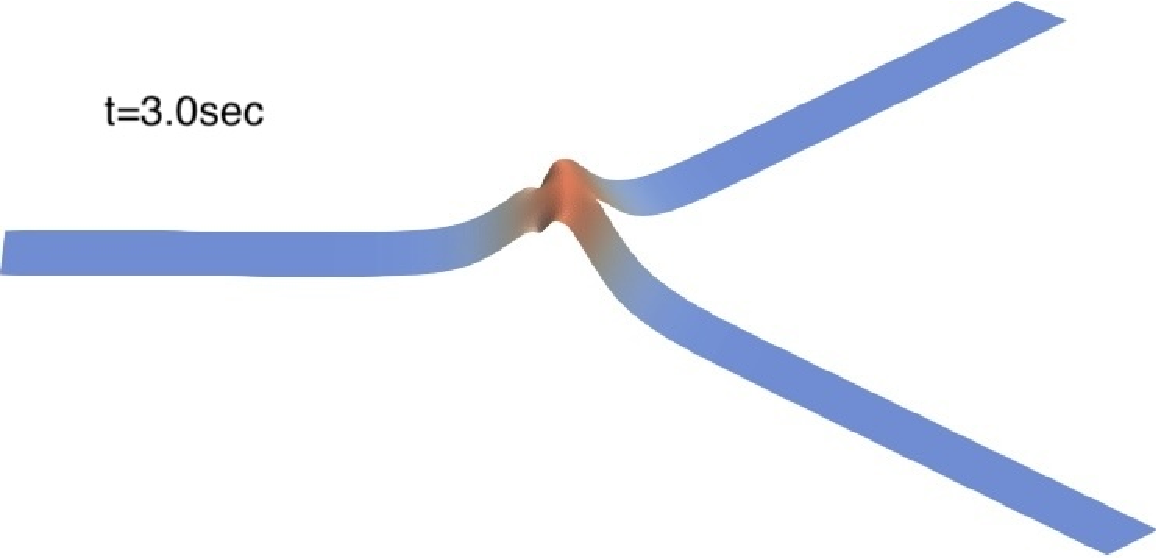}
\includegraphics[width=0.4\columnwidth]{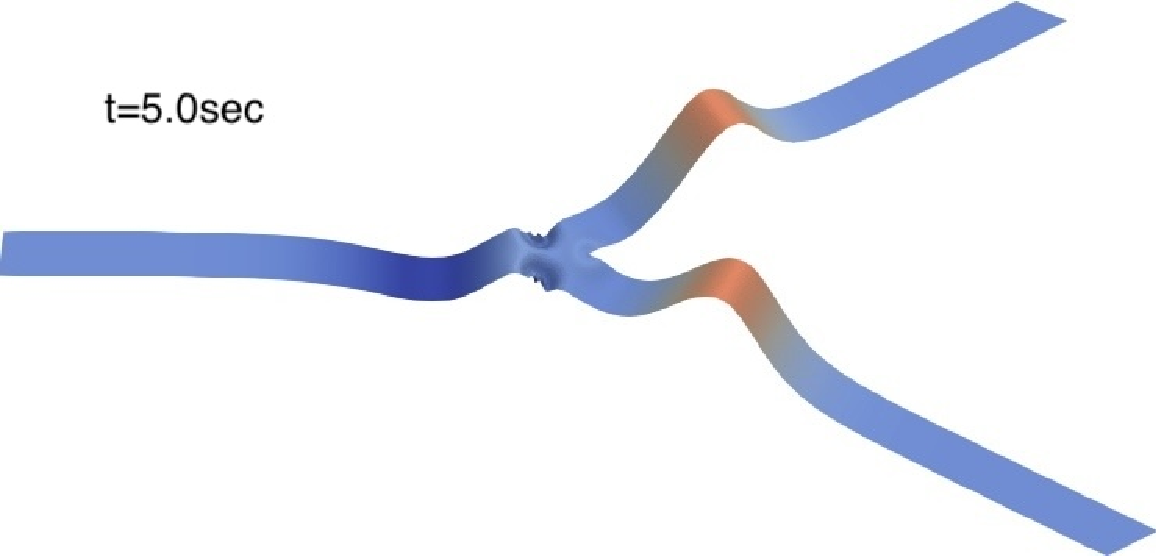}\\
\includegraphics[width=0.4\columnwidth]{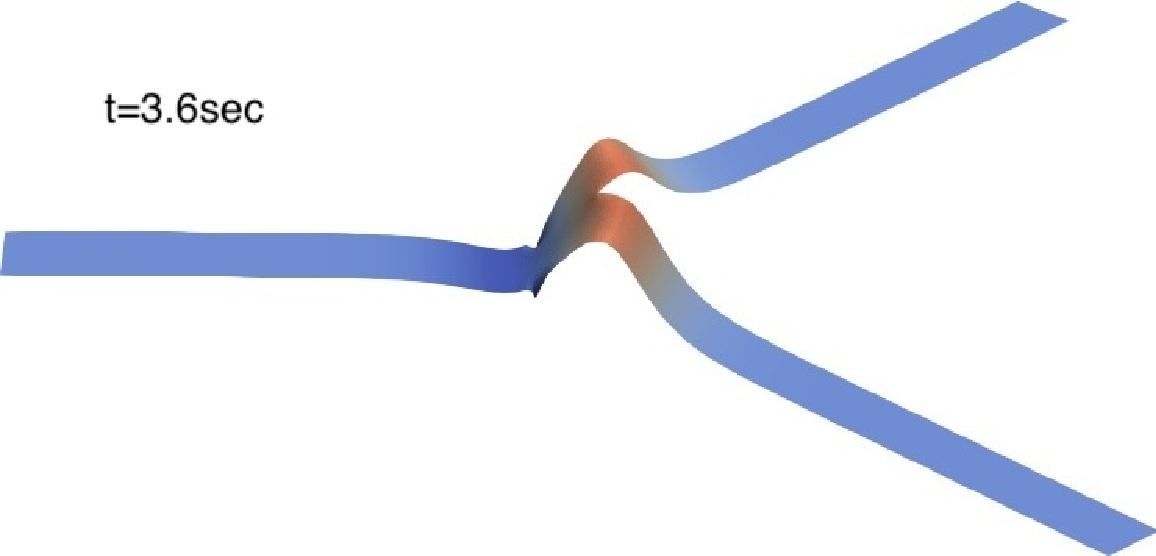}
\includegraphics[width=0.4\columnwidth]{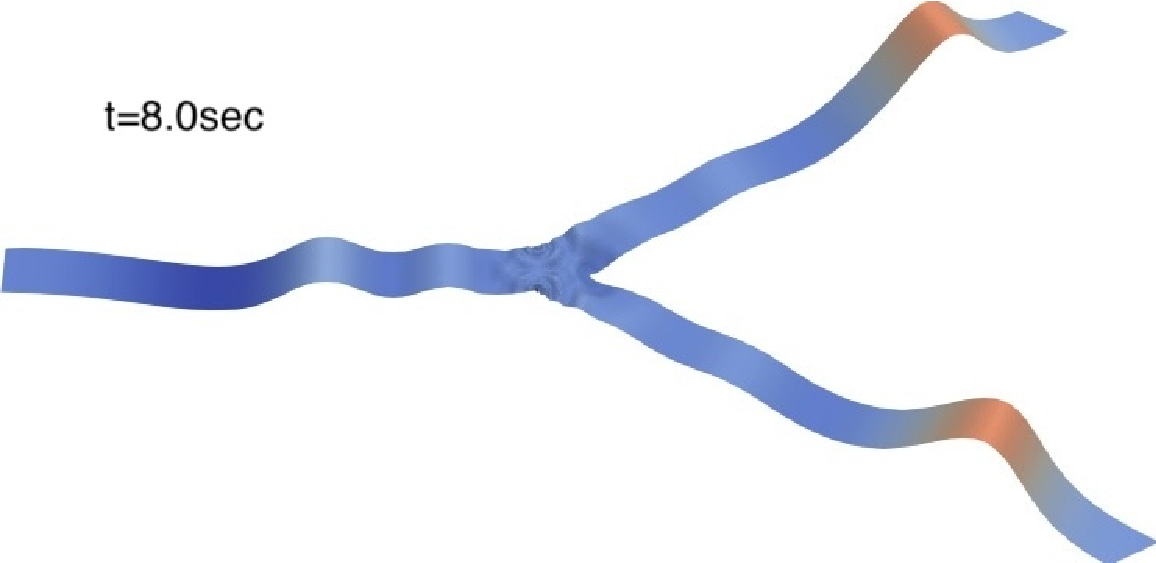}
\caption{Scattering of solitary wave at a Y-junction. BBM-BBM system with $\theta^2=2/3$}
\label{fig:bbmbbmj}
\end{figure}

\begin{figure}[t]
\centering
\includegraphics[width=1\columnwidth]{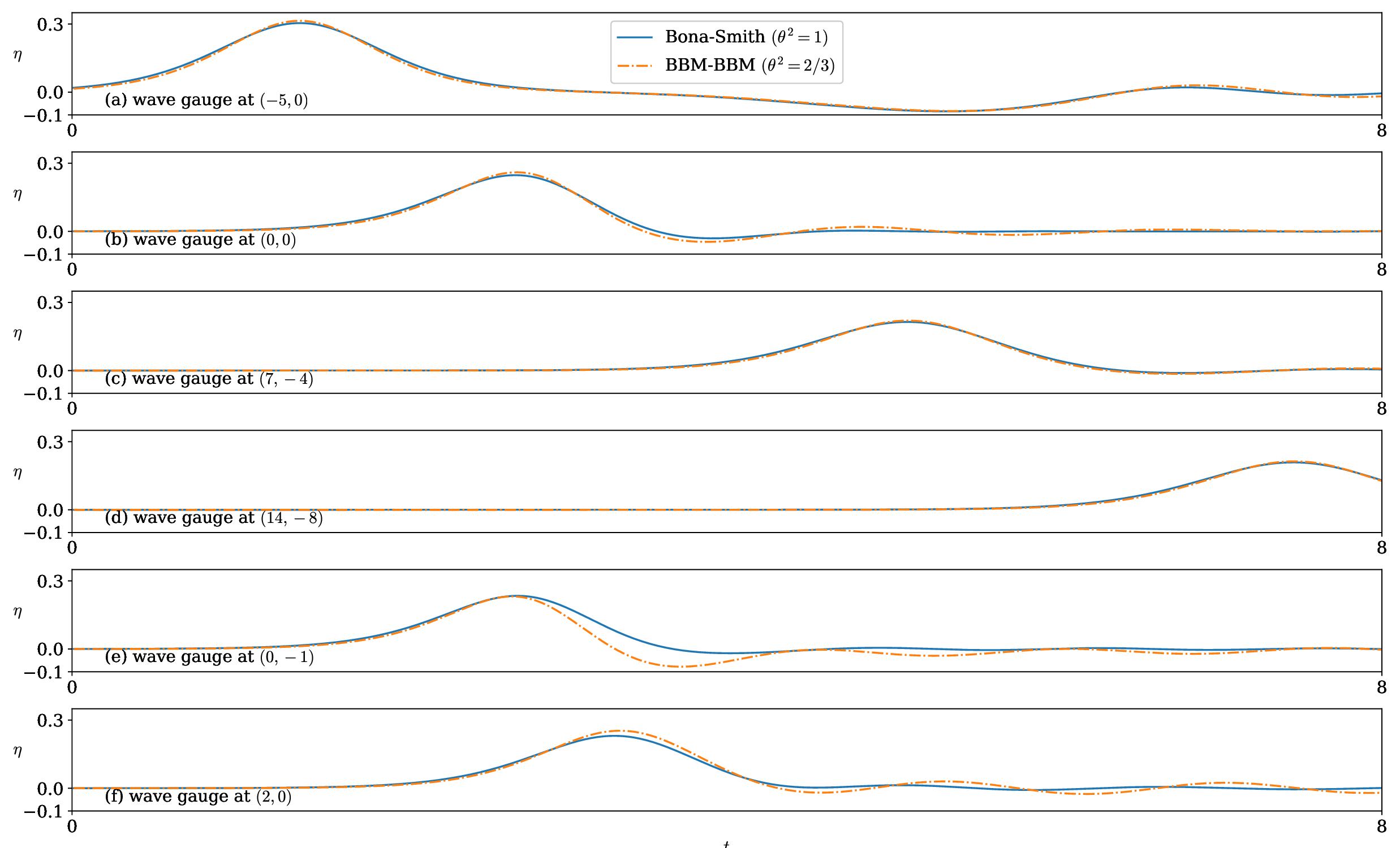}
\caption{Scattering of solitary wave at a Y-junction: Solution recorded at wave gauges}
\label{fig:jgauges}
\end{figure}

Specifically, we define the domain $\Omega$ as a polygon with vertices at $(-20.0, -1.0)$, $(0.0, -1.0)$, $(20.0, -12.5)$, $(22.0, -11.5)$, $(2.0, 0.0)$, $(22.0, 11.5)$, $(20.0, 12.5)$, $(0.0, 1.0)$, and $(-20.0, 1.0)$, as illustrated in Figure \ref{fig:domainj}, and a horizontal bottom with $D=1m$. We numerically solve problem (\ref{eq:BS2}) using a new conservative method with $\theta^2=1$, $g=9.81m/s^2$, and an initial condition of a numerical solitary wave with $c_s=3.6$ ($A\approx 0.3~m$). We also recorded the solution at six wave gauges located at $(-5,0)$, $(0,0)$, $(7,-4)$, $(14,-8)$, and at the corners $(0,-1)$ and $(2,0)$ represented by the red dots in Figure \ref{fig:domainj}. In this experiment, we employed $\mathcal{P}_h^2$ elements with a regular triangulation of $\Omega$ consisting of $19,036$ triangles. The maximum edge length recorded was $h\approx 0.2$. We also used $\Delta t=0.01$. The mass and energy remained constant (up to the digits shown) and were equal to $\mathcal{M}=2.9532623707529$ and $\mathcal{E}=6.50896080026$, respectively. 

Figure \ref{fig:junc} presents the numerical solution at times $t=0$, $2$, $3$, $4$, $5$, and $8$. We observed that the solution remained smooth at all times. Figure \ref{fig:jgauges} presents the solution recorded at the six wave gauges. The solitary wave travels unchanged until it reaches the junction corners $(0,-1)$ and $(0,1)$, where a significant portion of the solitary wave is reflected back as a dispersive wave train with a leading negative excursion. The remaining incident wave spreads into the two new directions due to diffraction. Upon interacting with the corner at $(2,0)$, it splits into two symmetric waves propagating along the connected channels. The maximum wave height recorded at the gauge $(14,-8)$ (at the far left-end of the channel) was $A\approx 0.21$, indicating that each of the transmitted waves is approximately two-thirds the size of the original wave. Similar results were observed in similar experiments in \cite{NDS2012}.

It is worth mentioning that we repeated this experiment with the BBM-BBM system ($\theta^2=2/3$), where $c=0$. In this experiment we generated a new solitary wave for the BBM-BBM system with the same speed as the Bona-Smith system. Although the two solitary waves are very similar, there are small differences. The lower regularity of the solutions in this case led to a hydraulic jump phenomenon at the corners of the junction, as presented in Figure \ref{fig:bbmbbmj}. This phenomenon can also be observed when solving the non-dispersive shallow water equations. Nevertheless, our numerical solution remained stable; mass, vorticity, and energy were conserved with values $\mathcal{M}=2.834206328927$, $\mathcal{V}=O(10^{-18})$, and $\mathcal{E}=6.1508487223870$, respectively. Furthermore, no spurious oscillations were observed, apart from significant reflections generated by the corners. In Figure \ref{fig:jgauges}, we observe that the solutions of the two systems match at the wave gauges in the interior of the domain. However, the gauges at the corners reveal the reflections generated in the numerical solution of the BBM-BBM system. Note that the proposed numerical method does not utilize limiters or artificial dissipative techniques; therefore, the stability observed in these experiments can be attributed primarily to the conservation properties of the new scheme.

It is noteworthy that the conservative Galerkin method performs exceptionally well in this specific experiment. However, further testing is needed for more general scenarios where wave-breaking occurs. In such cases, other numerical methods, such as discontinuous Galerkin methods and finite volume methods, are known to perform effectively (cf. e.g., \cite{DKM2011,PZR2021,KD2013,KDNS2012,BGSM2011,BDEFG2021}).

\section{Conclusions}\label{sec:conclusions}

In this work, our focus was on the Bona-Smith family of Boussinesq systems with variable bottom topography in bounded domains with slip-wall boundary conditions. We established the existence and uniqueness of smooth solutions to the corresponding initial-boundary value problem. This family of systems is specifically designed to preserve three fundamental for long waves invariants: mass, vorticity, and energy. Additionally, we developed a novel numerical method that maintains these three invariants. Spatial discretization is achieved through a modified Galerkin/finite element method, applied to both the equation of free-surface and the velocity potential, and implemented using a mixed formulation. Temporal discretization relies on explicit, relaxation Runge-Kutta methods, which are tailored to preserve the energy of the system. Mass conservation is intrinsic, due to its first-order functional relation to the free-surface elevation. The conservation of vorticity is ensured through the formulation of the problem using the velocity potential.

To assess the applicability of both the mathematical and numerical models, we conducted a series of numerical experiments. These experiments indicate that the numerical solutions converge with an optimal rate in space to both the surface elevation and the velocity potential. Furthermore, we validated the numerical solutions against laboratory data in three classical scenarios: the reflection of shoaling solitary waves by a vertical wall, the interaction of a solitary wave with a vertical cylinder and the interaction of a periodic wave-train with a submerged bar. In all cases, the agreement between the numerical and experimental data was satisfactory. Finally, we studied the scattering of a solitary wave by a Y-junction. This problem is especially challenging because it involves re-entrant corners. Overall, the stability and the accuracy of the new numerical model indicate excellent performance in simulations involving long waves of small amplitude even in domains with complicated geometric and bottom features.

\section*{Appendix: Derivation of the $abcd$-Boussinesq systems}

For the sake of self-containment we present in this appendix a derivation of the new $abcd$-Boussinesq systems. We consider $D(\bx)=D_0+D_b(\bx)$ the depth measured from the zero-level of the undisturbed free-surface of the water, where $D_0$ is a characteristic (mean) depth and $D_b$ the smooth variations of the bottom. If $\lambda_0$ is a typical wavelength, $a_0$ a typical wave height, and $d_0$ a typical order of bottom topography variations, then we consider the non-dimensional variables 
$\tilde{\bx}=\bx/\lambda_0$, $\tilde{z}=z/D_0$, $\tilde{t}= t c_0/\lambda_0$, 
$\tilde{\bu}=\bu D_0/(a_0c_0)$, $\tilde{w}=w \lambda_0/(a_0c_0)$, $\tilde{\eta}=\eta/a_0$, $\tilde{D}_b=D_b/d_0$, and $\tilde{p}=p(\rho g D_0)$,
where $c_0=\sqrt{gD_0}$ is the linear speed of propagation. As usual, we denote
$\varepsilon = a_0/D_0$, $\sigma=D_0/\lambda_0$, $\beta=d_0/D_0$, which are all assumed to be very small parameters.
Note that the scaled depth will be  $\tilde{D}=1+\beta \tilde{D}_b$ assuming a slowly varying bottom with $\beta\ll 1$.

Using these dimensionless variables, we write Euler's equations in non-dimensional form:
\begin{align}
& \tilde{\nabla} \cdot \tilde{\bu}+\tilde{w}_{\tilde{z}}=0\ , \label{eq:eul1}\\
&\varepsilon \tilde{\bu}_{\tilde{t}}+\varepsilon^2 [(\tilde{\bu}\cdot\tilde{\nabla})\tilde{\bu}+\tilde{w}\tilde{\bu}_{\tilde{z}}]+\tilde{\nabla}\tilde{p}=0\ ,  \label{eq:eul2}\\
&\varepsilon\sigma^2 \tilde{w}_{\tilde{t}}+\varepsilon^2\sigma^2 [\tilde{\bu}\cdot \tilde{\nabla} \tilde{w}+\tilde{w}\tilde{w}_{\tilde{z}}]+\tilde{p}_{\tilde{z}}=-1\ ,  \label{eq:eul3}
\end{align}
for $-\tilde{D}<\tilde{z}<\varepsilon\tilde{\eta}$. The irrotationality condition is written as:
\begin{align}
&\tilde{\nabla}\times \tilde{\bu}=0\ , \label{eq:eul4}\\
&\bu_z-\sigma^2\tilde{\nabla}\tilde{w}=0\ ,   \label{eq:eul5}
\end{align}
for $-\tilde{D}<\tilde{z}<\varepsilon\tilde{\eta}$. The boundary conditions on the free surface and the bottom can be expressed as
\begin{align}
& \teta_t+\varepsilon(\tilde{\bu}\cdot\tilde{\nabla}\teta)-\tilde{w}=0,\quad \tilde{p}=\frac{p_{\text{atm}}}{\rho g D_0}\quad \text{on}\quad \tilde{z}=\varepsilon\teta\ ,  \label{eq:eul6}\\
& \tilde{\zeta}_{\tilde{t}}+ \tilde{\bu}\cdot \tilde{\nabla}\tilde{D}+\tilde{w}=0\quad \text{on}\quad\tilde{z}=-\tilde{D}\ .  \label{eq:eul7}
\end{align}

Integrating the mass equation (\ref{eq:eul1}) between $-\tilde{D}$ and $\varepsilon\tilde{\eta}$, we obtain
\begin{equation}
\tilde{w}(\varepsilon\teta)-\tilde{w}(-\tilde{D})=-\int_{-\tilde{D}}^{\varepsilon\teta}\tilde{\nabla}\cdot\tilde{\bu}~d\tilde{z}\ .
\end{equation}
From the boundary conditions (\ref{eq:eul6}) and (\ref{eq:eul7}) we obtain the equation
\begin{equation}\label{eq:massper}
\teta_{\tilde{t}}+\tilde{\nabla}\cdot[(\tilde{D}+\varepsilon\tilde{\eta})\tilde{\bu}_a]=0\ ,
\end{equation}
where 
\begin{equation}\label{eq:dpav}
\tilde{\bu}_a(\tilde{\bx},\tilde{t})=\frac{1}{\tilde{D}+\varepsilon\teta}\int_{-\tilde{D}}^{\varepsilon\teta}\tilde{\bu}~d\tilde{z}\ ,
\end{equation}
denotes the depth-averaged horizontal velocity of the fluid. 

Integrating (\ref{eq:eul1}) from $-\tilde{D}$ to $\tilde{z}$, and using (\ref{eq:eul7}) we have
\begin{equation}\label{eq:eul8}
\tilde{w}=-\tilde{\bu}\cdot \tilde{\nabla}\tilde{D}-\int_{-\tilde{D}}^{\tilde{z}}\tilde{\nabla}\cdot\tilde{\bu}\ .
\end{equation}
After integration of (\ref{eq:eul5}) over $(0,\tilde{z})$, we have
$
\tilde{\bu}(\tilde{\bx},\tilde{z},\tilde{t})=\tilde{\bu}_0(\tilde{\bx},\tilde{t})+O(\sigma^2)\ ,
$
where $\tilde{\bu}_0(\tilde{\bx},\tilde{t})\doteq \tilde{\bu}(\tilde{\bx},0,\tilde{t})$. Thus, we have 
\begin{equation}\label{eq:eul9}
\tilde{\bu}(\tilde{\bx},\tilde{z},\tilde{t})=\tilde{\bu}_0(\tilde{\bx},\tilde{t})+O(\sigma^2)\ .
\end{equation}
Substitution of (\ref{eq:eul8}) into the irrotationality condition (\ref{eq:eul5}) and using (\ref{eq:eul9}) yields
\begin{equation}\label{eq:eul10}
\tilde{\bu}_{\tilde{z}}=-\sigma^2(\tilde{z}+1) \tilde{\nabla}(\tilde{\nabla}\cdot \tilde{\bu}_0)+O(\sigma^4,\varepsilon\sigma^2,\beta\sigma^2)\ .
\end{equation}
Integration of (\ref{eq:eul10}) from $0$ to $\tilde{z}$ implies
\begin{equation}\label{eq:eul11}
\tilde{\bu}=\tilde{\bu}_0-\sigma^2(\tilde{z}+\frac{\tilde{z}^2}{2})\tilde{\nabla}(\tilde{\nabla}\cdot \tilde{\bu}_0)+O(\sigma^4,\varepsilon\sigma^2,\beta\sigma^2)\ .
\end{equation}
Equation (\ref{eq:eul8}) using (\ref{eq:eul9}) and differentiating with respect to $t$ becomes
\begin{equation}\label{eq:eul12}
\tilde{w}=-\tilde{\nabla}\cdot(\tilde{D}\tilde{\bu}_0)-\tilde{z}\tilde{\nabla}\cdot \tilde{\bu}_0+O(\sigma^2)\ ,
\end{equation}
and differentiation with respect to $t$ yields
\begin{equation}\label{eq:eul13}
\tilde{w}_{\tilde{t}}=-\tilde{\nabla}\cdot(\tilde{D}\tilde{\bu}_0)_{\tilde{t}}-\tilde{z}\tilde{\nabla}\cdot {\tilde{\bu}_0}_{\tilde{t}}+O(\sigma^2)\ .
\end{equation}
We set $\tilde{P}=\tilde{p}-p_{\text{atm}}/\rho g D_0$ (so as $\tilde{\nabla} \tilde{P}=\tilde{\nabla} \tilde{p}$ and $\tilde{P}(\varepsilon\teta)=0$).
After integration of (\ref{eq:eul3}) from $\tilde{z}$ to $\varepsilon\tilde{\eta}$ we obtain the non-hydrostatic approximation for the pressure:
\begin{equation}\label{eq:eul14}
\tilde{P}=\varepsilon\sigma^2(\tilde{z}+\frac{\tilde{z}^2}{2})\tilde{\nabla}\cdot{\tilde{\bu}_0}_{\tilde{t}}+\varepsilon\tilde{\eta}-\tilde{z}+O(\varepsilon\sigma^4,\varepsilon^2\sigma^2,\varepsilon\beta\sigma^2)\ .
\end{equation}
Substitution of (\ref{eq:eul11}), (\ref{eq:eul12}) and (\ref{eq:eul14}) into (\ref{eq:eul2}) leads to the approximation of momentum conservation
\begin{equation}\label{eq:eul15}
{\tilde{\bu}_0}_{\tilde{t}}+\tilde{\nabla}\tilde{\eta}+\varepsilon(\tilde{\bu}_0\cdot\tilde{\nabla})\tilde{\bu}_0=O(\sigma^4,\varepsilon\sigma^2,\beta\sigma^2)\ .
\end{equation}
Equation (\ref{eq:eul11}) using (\ref{eq:dpav}) becomes
\begin{equation}\label{eq:eul17}
\tilde{\bu}_0=\tilde{\bu}_a-\frac{\sigma^2}{3}\tilde{\nabla}(\tilde{\nabla}\cdot {\tilde{\bu}_a})+O(\sigma^4,\varepsilon\sigma^2,\beta\sigma^2)\ .
\end{equation}
Since $\tilde{\bu}=\tilde{\bu}_a+O(\sigma^2)$, it is implied that $\tilde{\nabla}\times \tilde{\bu}_a=O(\sigma^2)$, which yields that $({\tilde{\bu}_a}\cdot\tilde{\nabla}){\tilde{\bu}_a}=\frac{1}{2}\tilde{\nabla}|\tilde{\bu}_a|^2+O(\sigma^2)$.
Subsequently, equation (\ref{eq:eul15}) yields the momentum equation
\begin{equation}\label{eq:eul19}
{\tilde{\bu}_a}_{\tilde{t}}+\tilde{\nabla}\tilde{\eta}+\frac{\varepsilon}{2}\tilde{\nabla}|\tilde{\bu}_a|^2-\frac{\sigma^2}{3}\tilde{\nabla}(\tilde{\nabla}\cdot{\tilde{\bu}_a}_{\tilde{t}})=O(\sigma^4,\varepsilon\sigma^2,\beta\sigma^2)\ .
\end{equation}
Note that equations (\ref{eq:massper})-(\ref{eq:eul19}) have been also studied in \cite{KMS2020}. Evaluating the horizontal velocity (\ref{eq:eul11}) at depth 
$\tilde{z}_\theta=-\tilde{D}+\theta(\varepsilon\tilde{\eta}+\tilde{D})$ for $0\leq \theta\leq 1$,
where $\sigma^2\tilde{z}_\theta=(\theta-1)\sigma^2+O(\varepsilon\sigma^2,\beta\sigma^2)$, and using (\ref{eq:eul17}) we obtain
\begin{equation}\label{eq:eul20}
\tilde{\bu}_\theta=\tilde{\bu}_a-\frac{\sigma^2}{2}\left(\theta^2-\frac{1}{3}\right)\tilde{\nabla}(\tilde{\nabla}\cdot\tilde{\bu}_a)+O(\sigma^4,\varepsilon\sigma^2,\beta\sigma^2)\ ,
\end{equation}
and equivalently
\begin{equation}\label{eq:eul21}
\tilde{\bu}_a=\tilde{\bu}_\theta+\frac{\sigma^2}{2}\left(\theta^2-\frac{1}{3}\right)\tilde{\nabla}(\tilde{\nabla}\cdot\tilde{\bu}_\theta)+O(\sigma^4,\varepsilon\sigma^2,\beta\sigma^2)\ .
\end{equation}
Substituting (\ref{eq:eul21}) into (\ref{eq:massper}) and (\ref{eq:eul19}) we obtain the system
\begin{align}
&\teta_{\tilde{t}}+\tilde{\nabla}\cdot[(\tilde{D}+\varepsilon\tilde{\eta})\tilde{\bu}_\theta]+\frac{\sigma^2}{2}\left(\theta^2-\frac{1}{3}\right)\tilde{\nabla}\cdot\tilde{\nabla}(\tilde{\nabla}\cdot\tilde{\bu}_\theta)=O(\sigma^4,\varepsilon\sigma^2,\beta\sigma^2)\ , \label{eq:eul22}\\
&{\tilde{\bu}_\theta}_{\tilde{t}}+\tilde{\nabla}\tilde{\eta}+\frac{\varepsilon}{2}\tilde{\nabla}|\tilde{\bu}_\theta|^2+\frac{\sigma^2}{2}(\theta^2-1)\tilde{\nabla}(\tilde{\nabla}\cdot{\tilde{\bu}_\theta}_{\tilde{t}})=O(\sigma^4,\varepsilon\sigma^2,\beta\sigma^2)\ . \label{eq:eul23}
\end{align}
Observe that from (\ref{eq:eul22}) and (\ref{eq:eul23}) we obtain
\begin{equation}\label{eq:eul24}
\tilde{\nabla}\cdot \tilde{\bu}_\theta=-\teta_{\tilde{t}}+O(\varepsilon,\beta,\sigma^2)\qquad
\text{and}\qquad
{\tilde{\bu}_{\theta}}_{\tilde{t}}=-\tilde{\nabla}\tilde{\eta}+O(\varepsilon,\sigma^2)\ . 
\end{equation}
Using the classical BBM-trick \cite{BBM1972} with (\ref{eq:eul24}) and taking arbitrary $\nu,\mu\in \mathbb{R}$ we write
\begin{equation}\label{eq:rel1}
\begin{aligned}
&\tilde{\nabla}\cdot\tilde{\nabla}(\tilde{\nabla}\cdot\tilde{\bu}_\theta)=\mu \tilde{\nabla}\cdot\tilde{\nabla}(\tilde{\nabla}\cdot\tilde{\bu}_\theta)-(1-\mu)\tilde{\nabla}\cdot\tilde{\nabla}~\teta_{\tilde{t}}+O(\varepsilon,\beta,\sigma^2)\quad \text{and}\\
&\tilde{\nabla}(\tilde{\nabla}\cdot{\tilde{\bu}_\theta}_{\tilde{t}})=-\nu\tilde{\nabla}(\tilde{\nabla}\cdot \tilde{\nabla}\tilde{\eta})+(1-\nu)\tilde{\nabla}(\tilde{\nabla}\cdot{\tilde{\bu}_\theta}_{\tilde{t}})+O(\varepsilon,\sigma^2)\ .
\end{aligned} 
\end{equation}
Substituting the relations (\ref{eq:rel1}) into the system (\ref{eq:eul22})--(\ref{eq:eul23}) we write the general $abcd$-Boussinesq system
\begin{align}
& \teta_{\tilde{t}}+\tilde{\nabla}\cdot[(\tilde{D}+\varepsilon\tilde{\eta})\tilde{\bu}_\theta]-\sigma^2 \tilde{\nabla}\cdot[a\tilde{\nabla}(\tilde{\nabla}\cdot\tilde{\bu}_\theta)+b\tilde{\nabla}\teta_{\tilde{t}}]=O(\sigma^4,\varepsilon\sigma^2,\beta\sigma^2)\ , \label{eq:eul26}\\
& {\tilde{\bu}_\theta}_{\tilde{t}}+\tilde{\nabla}\tilde{\eta}+\frac{\varepsilon}{2}\tilde{\nabla}|\tilde{\bu}_\theta|^2-\sigma^2 \tilde{\nabla}[ c \tilde{\nabla}\cdot \tilde{\nabla}\teta+d\tilde{\nabla}\cdot{\tilde{\bu}_\theta}_{\tilde{t}}]=O(\sigma^4,\varepsilon\sigma^2,\beta\sigma^2)\ .\label{eq:eul27}
\end{align}
where $a,b,c,d$ are given by the formulas (\ref{eq:abcdcoefs}). Furthermore, using the approximations 
$$
\begin{array}{ll}
\tilde{\nabla}(\tilde{\nabla}\cdot\tilde{\bu}_\theta)=\tilde{\nabla}(\tilde{D}^3\tilde{\nabla}\cdot\tilde{\bu}_\theta)+O(\varepsilon,\beta),
&\tilde{\nabla}\teta_{\tilde{t}}=\tilde{D}^2\tilde{\nabla}\teta_{\tilde{t}}+O(\varepsilon,\beta)\ , \\
\tilde{\nabla}\cdot \tilde{\nabla}\teta=\tilde{\nabla}\cdot (\tilde{D}^2 \tilde{\nabla}\teta)+O(\varepsilon,\beta), 
&\tilde{\nabla}\cdot{\tilde{\bu}_\theta}_{\tilde{t}}=\tilde{\nabla}\cdot (\tilde{D}^2{\tilde{\bu}_\theta}_{\tilde{t}})+O(\varepsilon,\beta)\ .
\end{array}
$$
we write the general $abcd$-Boussinesq system as
\begin{align}
& \teta_{\tilde{t}}+\tilde{\nabla}\cdot[(\tilde{D}+\varepsilon\tilde{\eta})\tilde{\bu}_\theta]-\sigma^2 \tilde{\nabla}\cdot[a\tilde{\nabla}(\tilde{D}^3\tilde{\nabla}\cdot\tilde{\bu}_\theta)+b\tilde{D}^2\tilde{\nabla}\teta_{\tilde{t}}]=O(\sigma^4,\varepsilon\sigma^2,\beta\sigma^2)\ , \label{eq:eul28}\\
& {\tilde{\bu}_\theta}_{\tilde{t}}+\tilde{\nabla}\tilde{\eta}+\frac{\varepsilon}{2}\tilde{\nabla}|\tilde{\bu}_\theta|^2-\sigma^2 \tilde{\nabla} [ c \tilde{\nabla}\cdot (\tilde{D}^2 \tilde{\nabla}\teta)+d\tilde{\nabla}\cdot (\tilde{D}^2{\tilde{\bu}_\theta}_{\tilde{t}})]=O(\sigma^4,\varepsilon\sigma^2,\beta\sigma^2)\ .\label{eq:eul29}
\end{align}

The justification of these systems against the Euler equations was presented in \cite{IKM2023} for large time intervals of order $1/\varepsilon^2$. Discarding the high-order terms in (\ref{eq:eul28})--(\ref{eq:eul29}) and transforming the remaining equations back to dimensional variables we obtain the system (\ref{eq:abcd}).

\section*{Declarations}

\subsection*{Conflict of interest} 

On behalf of all authors, the corresponding author states that there is no conflict of interest.

\subsection*{Data availability}

Data sharing not applicable to this article as no datasets were generated or analysed during
the current study.


\end{document}